\documentclass{amsart}

\usepackage{amssymb}
\usepackage{amsthm}
\usepackage{amsmath}
\usepackage{comment}
\usepackage{tikz-cd}

\usepackage{ stmaryrd }
\usepackage{tikz}
\usetikzlibrary{arrows,calc}

\usepackage[foot]{amsaddr}

\usepackage[colorlinks = true,backref = page,hyperindex,breaklinks]{hyperref}

\theoremstyle{plain}
\newtheorem{proposition}{Proposition}[section]
\newtheorem{theorem}[proposition]{Theorem}
\newtheorem{lemma}[proposition]{Lemma}
\newtheorem{corollary}[proposition]{Corollary}
\theoremstyle{definition}

\newtheorem{definition}[proposition]{Definition}

\theoremstyle{remark}
\newtheorem{remark}[proposition]{Remark}

\renewcommand{\hbar}{\bar{{\mathbb H}}^3}

\newcommand{\CC}{\mathbb C}

\newcommand{\R}{\mathbb R}

\newcommand{\Half}{{\mathbb H}}
\newcommand{\Hp}{{\mathbb H}^2}
\newcommand{\Hs}{{{\mathbb H}^3}}

\newcommand{\psl}{{\rm PSL}(2,\CC)}

\newcommand{\Ad}{\operatorname{Ad}}
\newcommand{\QF}{\operatorname{QF}}

\newcommand{\Id}{{\operatorname{Id}}}

\newcommand{\supp}{{\operatorname{supp}}}

\newcommand{\sech}{{\mathrm{sech}}}
\newcommand{\cartan}{a}

\newcommand{\PSL}{\mathsf{PSL}}
\newcommand{\SL}{\mathsf{SL}}
\newcommand{\GL}{\mathsf{GL}}
\renewcommand{\P}{{\sf{Pr}}}

\newcommand{\aff}[1]{\sf{F}_{#1}}
\renewcommand{\sf}[1]{{\mathsf{#1}}}
\newcommand{\cal}{\mathcal}
\newcommand{\lb}{\llbracket}
\newcommand{\rb}{\rrbracket}

\newcommand{\h}{{\mathfrak h}}
\newcommand{\C}{\mathbb C}
\renewcommand{\sl}{\mathfrak{sl}}
\newcommand{\diag}{{\operatorname{diag}}}

\newcommand{\G}{\Gamma}
\newcommand{\g}{\gamma}
\newcommand{\M}{\mathrm{MS}}
\newcommand{\co}{\sf u}
\renewcommand{\C}{\mathbb C}
\newcommand{\ad}{\operatorname{ad}}
\renewcommand{\a}{\frak a}
\newcommand{\eps}{\varepsilon}
\newcommand{\Hom}{\mathrm{Hom}}

\newcommand{\dd}{{\mathrm{d}}}

\begin{document}

\title[Bending, entropy and proper affine actions]{Bending, entropy and proper affine actions of surface groups}
\author[Bridgeman]{Martin Bridgeman}
\address{Boston College}
\author[Canary]{Richard Canary}
\address{University of Michigan}
\author[Sambarino]{Andres Sambarino}
\address{CNRS - Universit\'e Sorbonne Paris Nord}
\thanks{Bridgeman was partially supported by grant DMS-2405291 from the National Science Foundation and Simons Foundation Travel Grant MPS-TSM-00007295.
Canary was partially supported by grant  DMS- 2304636 from the National Science Foundation. The authors acknowledge support of the Institut Henri Poincar\'e 
(UAR 839 CNRS-Sorbonne Université), and LabEx CARMIN (ANR-10-LABX-59-01).
Bridgeman and Canary were partially supported by the National Science Foundation
under Grant No. DMS-1928930, while they were in residence at the Simons Laufer Mathematical Sciences Institute in Berkeley, California, during the Spring 2026 semester.}

\date{\today}

\begin{abstract} 
We show that for any closed surface $S$ there is an explicit neighborhood $V$ of the fuchsian locus in quasifuchsian space $\QF(S)$ such that for every representation 
$\rho\in V$ which is not
fuchsian,   there is a proper affine action on $\mathfrak{sl}(2,\CC)$ with linear part $\Ad(\rho)$. We further show that there is a larger neighborhood  
$U$ of the Fuchsian locus so that every
critical point of the entropy function in $U$ lies on the Fuchsian locus.
\end{abstract}

\maketitle

\setcounter{tocdepth}{1}
\tableofcontents

\section{Introduction}

Quasifuchsian hyperbolic 3-manifolds are a central object of study in low-dimensional topology and dynamics. One classical theme 
is the relationship between the geometry of the convex core and dynamical quantities like the topological entropy of the geodesic flow.
More recently, connections 
have emerged between quasifuchsian manifolds and proper affine actions of surface groups on $\sl(2,\C)\cong\mathbb R^6$.
In this paper, we investigate  bending deformations on the space $QF(S)$ of all (marked) quasifuchsian hyperbolic 3-manifolds.
We introduce the notion of a moderately bent Jordan domain and obtain applications to both the variation of the entropy function on $QF(S)$
and to proper affine actions of surface groups.

We begin our investigation by using work  Kourouniotis \cite{kouron,kouron-cont} to establish
a general formula for the variation  of the complex length function of an element  of $\pi_1(S)$ during a bending deformation.
If a boundary component of the convex core is moderately bent, we obtain control on the real part of the variation as
we bend along the associated bending lamination.

Our first application of this work is to show that if one component of the boundary of the  convex core of a quasifuchsian hyperbolic 3-manifold
(which is not fuchsian) is moderately bent, then it is not a critcal point of the entropy function on $QF(S)$.
Recall that the topological entropy of the geodesic flow of a quasifuchsian hyperbolic 3-manifold is the exponential growth
rate of the number of closed geodesics of length at most $T$. The topological entropy is an analytic function on $\QF(S)$ (see Ruelle \cite{ruelle}) and
it achieves its minimum value, which is 1, exactly along the locus of fuchsian groups (see Bowen \cite{bowen}). Previous work of Bridgeman \cite{bridgeman-wp}
provided some evidence that every critical point of the entropy function lies on the Fuchsian locus.

Our second application is to show that  if $\rho:\pi_1(S)\to\PSL(2,\C)$ is the holonomy representation of a quasifuchsian hyperbolic
3-manifold (which is not fuchsian) and both components of the boundary of the convex core are moderately bent, then $\Ad\rho$
is the linear part of a proper affine action of $\pi_1(S)$ on $\sl(2,\C)$. Danciger, Gu\'eritaud and Kassel \cite{DGK-aff} previously
exhibited  specific quasifuchsian hyperbolic 3-manifolds with this property. We extend our results to find proper affine actions of surface groups
on any simple complex Lie group. We also give applications to proper actions of surface groups on the group manifold of a complex simple Lie group.

Finally, we use work of Bridgeman, Canary and Yarmola \cite{BCY} to find explicit bounds on the ``roundness'' of the bending
lamination which guarantee that  the associated boundary component of the convex core is moderately bent. We can then describe
an explicit open neighborhood $U$ of the Fuchsian locus in $QF(S)$ so that every manifold in $U$ which is not fuchsian is not a
critical point of the entropy function. Similarly, we describe
an explicit open neighborhood $V$ of the Fuchsian locus in $QF(S)$ so that if $\rho$ is the holonomy representation of  manifold in $V$ which is not fuchsian,
then $\Ad\rho$ is the linear part of a proper affine action of $\pi_1(S)$ on $\sl(2,\C)$.

We next develop the language to state our results more concretely.  We follow this with a more in-depth discussion of the history of
proper affine actions.

\medskip\noindent
{\bf Statement of results:}
A  {\em quasifuchsian} representation is a  discrete faithful representation $\rho:\pi_1(S)\to \mathsf{PSL}(2,\mathbb C)$, 
where $S$ is a closed oriented surface of genus at least two, such that the limit set 
$\Lambda(\rho)$ of $\rho(\pi_1(S))$  is a Jordan curve. {\em Quasifuchsian space} $\QF(S)$ is the space of conjugacy classes of
quasifuchsian representations. We may regard
$\QF(S)$ as an open subset of the character variety
$$\frak X(S)=\mathrm{Hom}(\pi_1(S),\mathsf{PSL}(2,\mathbb C))//\mathsf{PSL}(2,\mathbb C).$$
Since every point in $QF(S)$ is a smooth point of $\frak X(S)$ (see \cite{goldman-symplectic}),
$QF(S)$ has the structure of a complex manifold. 

If $\rho$ is quasifuchsian then the complement of the limit set $\Lambda(\rho)$ is two Jordan domains $\Omega_+(\rho)$ and $\Omega_-(\rho)$.
The group $\rho(\pi_1(S))$ acts properly discontinuously on $\hat\C-\Lambda(\rho)$, so one obtains two marked Riemann
surfaces $X_+(\rho)=\Omega_+(\rho)/\rho(\pi_1(S))$ and $X_-(\rho)=\Omega_-(\rho)/\rho(\pi_1(S))$ where we  choose the labels on 
the components of $\hat\C-\Lambda(\rho)$
so that there is an orientation-preserving  homeomorphism from $S\cup\bar S\to X_+(\rho)\cup X_-(\rho)$ in the homotopy class determined
by $\rho$ (where $\bar S$ is $S$ with the opposite orientation). Bers showed that the resulting map from  $\QF(S)$ to 
$\mathcal T(S)\times\mathcal T(\bar S)$ is a real analytic diffeomorphism,
where $\mathcal T(S)$ is the Teichm\"uller space of $S$. 
A representation $\rho\in \QF(S)$ is {\em fuchsian} if it is conjugate to a representation with image in $\mathsf{PSL}(2,\R)$.
The {\em Fuchsian locus} $\mathrm{F}(S)\subset \QF(S)$ of fuchsian representations is identified with the diagonal in
this parametrization.

If  $\Omega$ is a Jordan domain whose boundary is a Jordan curve $C$,
we say a pair $(x,y)$ of distinct points in $C$ is a {\em bending pair} for $\Omega$ if there is a round open disk $D$ in $\Omega$ such that 
$x,y\in \partial D$.  We say that $\Omega$ is {\em moderately bent} if for every bending pair $(x,y)$ there exists a circle $L$ transverse to 
$C$ such that $L\cap C = \{x,y\}$. 

\medskip

We define the {\em entropy function} $h:\QF(S)\to [1,2)$  by letting $h([\rho])$ be the topological entropy of the geodesic flow on 
$N_\rho=\mathbb H^3/\rho(\pi_1(S))$, i.e.
the exponential growth rate of the number of closed geodesics in $N_\rho$ of length at most $T$. 
Sullivan \cite{sullivan-hd} proved that $h([\rho])$ is the Hausdorff dimension
of the limit set $\Lambda(\rho)$ of $\rho(\pi_1(S))$. One may also define 
$h([\rho])$ to be the {\em critical exponent} of the {\em Poincar\'e series}
$$P_\rho(s)  =\sum_{\gamma \in \pi_1(S)} e^{-sd(x_0,\gamma (x_0))},$$
for any $x_0\in\mathbb H^3$, i.e. $P_\rho(s)$ converges if $s>h([\rho])$ and diverges if $s<h([\rho])$.

Ruelle \cite{ruelle} showed that $h$ is real analytic. Bowen \cite{bowen} showed that $h$ attains its minimum value
exactly along the Fuchsian locus. One might hope that $h$ has no other critical points. Bridgeman \cite{bridgeman-wp}
showed that  the Hessian of the entropy function at any critical point is positive definite on at least a half-dimensional subspace, so the entropy 
functional has no local maxima. 

We show that if one of the components of the complement  of the limit set of a quasifuchsian (but not fuchsian) representation $\rho$
is moderately bent, then $\rho$ is not a critical point of the entropy function.

\begin{theorem}
\label{Main Theorem 2}
Suppose that  $[\rho]\in \QF(S)$ is  not fuchsian and either $\Omega_+(\rho)$ or $\Omega_-(\rho)$ is moderately bent,
then $[\rho]$ is not a critical point of the entropy function $h$.
\end{theorem}

\medskip

We recall that if $V$ is a finite-dimensional vector space, then the set  $\mathrm{Aff}(V)$ of affine transformations of $V$  is 
the semi-direct product of $\mathsf{GL}(V)\ltimes V$. The action of $(A,w)\in \mathrm{Aff}(V)$ is given by
$$(A,w)(v)=Av+w\quad\text{for all}\quad v\in V.$$
Let $\Ad:\mathsf{PSL}(2,\mathbb C)\to\SL(\mathfrak{sl}(2,\CC))\subset\mathrm{Aff}(\sl(2,\C))$ be the adjoint representation,
i.e. $\Ad(A)(v)=AvA^{-1}$ for all $v\in\sl(2,\C)$. If $\sigma:\Gamma\to\mathrm{Aff}(V)$ is a representation, then the restriction of
$\sigma$ to the first factor $\GL(V)$ is a representation which we call the {\em linear part} of $\sigma$.

\begin{theorem}
\label{Main Theorem 3}
If $[\rho]\in \QF(S)$ is not fuchsian and both $\Omega_+(\rho)$ and $\Omega_-(\rho)$ are moderately bent,
then there exists a representation $\sigma:\pi_1(S)\to\mathrm{Aff}(\sl(2,\mathbb C))$  whose linear part is $\Ad(\rho)$
so that $\sigma(\pi_1(S))$ acts properly discontinuously on $\sl(2,\C)$.
\end{theorem}

The proofs of both Theorem \ref{Main Theorem 2} and Theorem \ref{Main Theorem 3} involve studying bending deformations of
quasifuchsian representations.
If $\rho$ is quasifuchsian, then its {\em convex core} $C(\rho)$ is the quotient of the convex hull $CH(\rho)$ of $\Lambda(\rho)$ in $\mathbb H^3$ by
$\rho(\pi_1(S))$. If $\rho$ is not fuchsian, then $C(\rho)$ is homeomorphic to $S\times [0,1]$ and
has two boundary components $\partial C_\pm(\rho)$,
each of which is totally geodesic in the complement of a geodesic lamination $\beta_\pm$. 
We choose our signs so that $\partial CH_+(\rho)$ faces $\Omega_+(\rho)$. We often abuse notation by saying that $\partial  C_\nu(\rho)$
is moderately bent when $\Omega_\nu(\rho)$ is moderately bent.
The intrinsic metric on each component of $\partial C(\rho)$ is hyperbolic, and one may associate  a transverse measure
to $\beta_\pm$ which measures the total bending  of $\partial C_\pm(\rho)$ along the lamination, to obtain the 
{\em bending laminations} (see Epstein-Marden \cite{EM} for background on convex hulls). Dular and Schlenker \cite{dular-schlenker} 
recently showed that $\rho$ is completely determined by
its pair of bending laminations.

If $\rho$ is quasifuchsian and $\lambda$ is a measured lamination, one may define a {\em bending deformation}
$\{\rho_{z\lambda}\}_{z\in\C}\subset \frak X(S)$ of $\rho$ along $\lambda$.
If $\lambda$ is a simple closed separating curve $a$, then we may write
$\pi_1(S)=\pi_1(S_1)*_{\langle \alpha\rangle}\pi_1(S_2)$ and define $\rho_{z\lambda}$ for any $z\in\C$ by letting
$\rho_{z\lambda}|_{\pi_1(S_1)}=\rho|_{\pi_1(S_1)}$ and $\rho_{z\lambda}(g)=A_z\rho(g)A_z^{-1}$ for all $g\in\pi_1(S_2)$ 
where $A_z$ has the same axis as $\rho(g)$ and complex translation length $z$.
In this convention, pure twisting corresponds to the case when $z$ is totally real and pure bending corresponds 
to the case where $z$ is totally imaginary.
We  generalize work of Kourouniotis \cite{kouron,kouron-cont} to obtain a formula for
the derivative of the complex length of any element of $\pi_1(S)$, see Theorem \ref{derivative of complex length}. 
The bending deformation is the complex Hamiltonian of the complex length function $\mathcal L_\lambda:\QF(S) \rightarrow \C$ 
of the lamination $\lambda$ with respect to the natural complex symplectic structure on $\QF(S)$ described 
by Goldman (see \cite{goldman-qc-symp,goldman-symplectic}), see Platis \cite{platis}.

We  generalize work of Kourouniotis \cite{kouron,kouron-cont} to obtain a formula for
the derivative of the complex length of an element $\gamma\in\pi_1(S)$ with respect to a bending deformation
along a bending lamination $\beta_\nu$ for $\rho\in QF(S)$, see Theorem \ref{derivative of complex length}. 
Our formula is expressed in terms of the complex distances between the axis of $\rho(\gamma)$ and lifts
of leaves of $\beta_n$. Assuming that $\Omega_\nu(\rho)$ is moderately bent places restrictions
on these complex distances.

If  $\rho$ is quasifuchsian and $\beta_\nu$ is a bending lamination of $\rho$ (where $\nu\in\{\pm\}$) and
$\{\rho_t\}$ is the bending deformation of $\rho_0$ along  $-i\beta_\nu$, then we define the {\em infinitesmal bending deformation} along
$-i\beta_\nu$ to be
$$w_\nu(\rho)=\frac{d}{dt}\Big|_{t=0}\rho_{-ti\beta_\nu}\in \sf{T}_{[\rho]}\QF(S).$$
(As defined, $w_\nu(\rho)$ lies in $\sf{T}_\rho\mathrm{Hom}(\pi_1(S),\PSL(2,\C))$, but it descends to an element of $\sf{T}_{[\rho]}\QF(S)$ which is
independent of the choice of representative of $[\rho]$.)
We choose $-it\beta_\nu$ as our deformation since this corresponds to increasing the bending measure of the bending lamination on the convex core boundary. 
If $\gamma\in\pi_1(S)$, consider the real analytic function \hbox{$\ell_\gamma:\QF(S)\to (0,\infty)$} where
$\ell_\gamma([\rho])$ is the (real) translation distance of $\rho(\Gamma)$. 
We say that a closed  curve intersects a bending lamination $\beta_\nu$ transversely if its geodesic representative in $\partial C_\nu(\rho)$ intersects $\beta_\nu$
transversely.

\begin{theorem}
\label{Main Theorem 1}
Suppose that $[\rho]\in \QF(S)$ is not fuchsian, $\Omega_\nu(\rho)$ is moderately bent (for some $\nu\in\{\pm\}$) and
$w=w_\nu(\rho)$ is the infinitesimal bending deformation for $-i\beta_\nu$.
If $\gamma\in \pi_1(S)$, then $d\ell_\gamma(w)\le 0$ and $d\ell_\gamma(w) <0$ if  $\gamma$ intersects $\beta_\nu$ transversely.
\end{theorem}

Theorem \ref{Main Theorem 2} then follows immediately from Theorem \ref{Main Theorem 1} and a result of Sambarino \cite[Lemma 2.32]{sambarino},
see Section \ref{entropy}.

\medskip

In order to prove Theorem \ref{Main Theorem 3} we need the following strengthening of Theorem \ref{Main Theorem 1}.

\begin{theorem}\label{Main Theorem 4}
Suppose that $[\rho]\in \QF(S)$ is not fuchsian and that $w= w_+(\rho)+w_-(\rho)$. If both $\Omega_+(\rho)$ and $\Omega_-(\rho)$ are
moderately bent, then there exists  $K > 0$ such that
$$\dd\ell_\g(w) \leq  -K \ell_\g(\rho)$$
for all $\gamma\in\pi_1(S)$.
\end{theorem}

We now outline how Theorem \ref{Main Theorem 4} is used to prove Theorem \ref{Main Theorem 3}.
If $\rho:\pi_1(S)\to\PSL(2,\C)$ is quasifuchsian, then $\co:\pi_1(S)\to\sl(2,\C)$ is a cocycle for $\Ad\rho$ if
$$\co(\g\eta)=\co(\g)+\Ad\rho(\g)\co(\eta)\quad\text{for all}\quad \gamma,\eta\in\pi_1(S).$$
Cocycles for $\Ad\rho$ are in one-to-one correspondence with Affine representations with linear part $\Ad\rho$. Specifically, for $\co:\pi_1(S)\to\sl(2,\C)$, 
the map $F_{\rho,\co}:\pi_1(S)\to\mathrm{Aff}(\sl(2,\C))$ given by
$$F_{(\rho,\co)}(\g)=(\Ad\rho(\g),\co(\g))\quad \mbox{for all} \quad \g\in\pi_1(S)$$
is an affine representation if and only if $u$ is a cocycle for $\Ad\rho$. 

Suppose that $[\rho]\in QF(S)$ and $v\in \sf T_{[\rho]} QF(S)$, then we choose a vector
$\tilde v\in  \sf T_{\rho} \mathrm{Hom}(\pi_1(S),\PSL(2,\CC))$ which projects to $v\in  \sf  T_{[\rho]} QF(S)$.
If $\gamma\in \pi_1(S)$, then the function $A_\gamma:\mathrm{Hom}(\pi_1(S),\PSL(2,\CC))\to\PSL(2,\C)$ given by $A_\gamma(\rho)=\rho(\gamma)$
is a complex analytic map. We define 
$$\co_{\tilde v}:\pi_1(S)\to\sl(2,\C)\quad\text{by letting}\quad \co_{\tilde v}(\gamma)=dA_\gamma(\tilde v) A_\gamma(\tilde\rho)^{-1}$$
for all $\gamma\in\pi_1(S)$, which is a cocycle for $\Ad\rho$. Cocycles for $\Ad\rho$ are in one-to-one correspondence
with tangent vectors in $ \sf T_\rho \mathrm{Hom}(\pi_1(S),\PSL(2,\CC))$. Different choices of representatives of $[\rho]$ and/or $\tilde v$ will produce
representations conjugate to $F_{(\rho,\co_{\tilde v})}$.

Kassel and Smilga \cite{KS}, see also Ghosh \cite{Ghoshdef}, developed a properness criterion for
affine actions which is expressed in terms of the Margulis invariant spectrum, see Proposition \ref{notzero}.
Theorem \ref{Main Theorem 4} will allow us to control the Margulis invariant spectrum of 
$(\rho,\co_w)$ when $\rho$ and $w$ satisfy the assumptions of the theorem. 
See Section \ref{affine actions} for details.

\medskip

In order to obtain the explicit neighborhoods of the Fuchsian locus that we promised earlier we recall a numerical
invariant which we can use to guarantee that $\Omega_\pm$ is moderately bent.
If $\beta_\nu$ is a bending lamination and $L>0$, we let $\|\beta_\nu\|_L$ be the supremum of the measure of any half-open arc of length less than $L$. 
The invariant $\|\beta_\nu\|_L$ is sometimes called the $L$-roundness of $\beta$ and was originally introduced by
Epstein, Marden and Markovic \cite{EMM}.

We use the techniques of Bridgeman-Canary-Yarmola \cite{BCY} to show that there exist bounds on the $L$-roundness of $\beta_\nu$
which guarantee that $\Omega_\nu$ is moderately bent.
We define  a threshold function  $r:\R_+ \rightarrow \R_+$ such that if $\|\beta_\nu\|_L < r(L)$ the  region $\Omega_\nu$ is moderately bent.
On $(0,1]$, $r$ is the inverse of the function $y =x\sec(x)$ and for $x > 1$, $r(x) =x\sech(x)$. 
A simple analysis of the function $x\sech(x)$ shows that $r(x) \geq x\sech(x)$ for all $x$.

\begin{theorem}
\label{Main Theorem 5}
Let $\rho\in \QF(S)$ have bending laminations $\beta_+$ and $\beta_-$.  If 
$$\|\beta_\nu\|_L<r(L)\quad\text{for some}\quad L>0,$$
then $\Omega_\nu(\rho)$ is moderately bent.
\end{theorem}

Our promised neighborhoods now have the following concrete form.

\begin{corollary}
\label{neighborhoods}
Let 
$$U(S)=\{[\rho]\in \QF(S): \|\beta_+\|_{L_1}<r(L_1)\ \ \text{for some}\ \ L_1>0\ \ \text{or}\quad  \|\beta_-\|_{L_2}<r(L_2)\ \ \text{for some}\ \ L_2>0\}.$$
and
$$V(S)=\{[\rho]\in \QF(S): \|\beta_+\|_{L_1}<r(L_1)\ \ \text{for some}\ \ L_1>0\ \ \text{and}\quad  \|\beta_-\|_{L_2}<r(L_2)\ \ \text{for some}\ \ L_2>0\}.$$
\begin{enumerate}
\item
$U(S)$ is an open neighborhood of the Fuchsian locus, so that if $[\rho]\in U(S)$ is not fuchsian, then $[\rho]$ is not a critical point
of the entropy function $h$.
\item
$V(S)$ is an open neighborhood of the Fuchsian locus, so that if $[\rho]\in V(S)$ is not fuchsian, then $\Ad\rho$ is the linear part
of a proper affine action of $\pi_1(S)$ on $\sl(2,\C)$.
\end{enumerate}
\end{corollary}

We now discuss the explicit constants we obtain. 
For example, 
$$r(1) \approx .739$$
For comparison, it follows from \cite[Thm. 4.1]{BCY} that for any $\rho\in \QF(S)$ one has
$$\|\beta_\pm\|_L\le 2\cos^{-1}\left(-\sinh(L/2)\right).$$
In particular,
$$\|\beta_\pm\|_1\le 2\cos^{-1}(-\sinh(1/2))\approx 4.2379.$$
To consider how close $r(L)$ is to being maximal, we describe the following elementary example. 
Take a horocycle $H$ in a hyperbolic plane $P \subseteq \Hs$ and let $\gamma:\R\rightarrow P$ be a piecewise geodesic with vertices on $H$ 
and each edge of length $L$. We then let $C$ be the convex pleated plane containing $\gamma$ and perpendicular to $P$. 
Then a simple calculation gives that $C$ has bending lamination $\beta$ with 
$$\|\beta\|_L = 2\sin^{-1}\left(\tanh(L/2)\right).$$
It follows that  there are non-embedded convex pleated planes with
$ \|\beta\|_1$ arbitrarily close to   $2\sin^{-1}\left(\tanh(1/2)\right)\simeq .9607.$

In Section \ref{classical} we discuss other bounds guaranteeing that our domains of discontinuity are moderately bent.
These bounds are in terms of the Schwarzian derivative, the Teichm\"uller distance between the two boundary components
and the quasiconformal distortion of $\Lambda(\rho)$.

\medskip

One may also obtain the following applications to more general complex Lie groups. If $\sf G$ is a complex simple Lie group and $\Gamma$ is a finitely generated group,
let $\mathcal C(\G,\sf G)$ be the space of pairs $(\rho,\co)$ where $\rho:\Gamma\to \sf G$ is a representation and $\co$ is a cocycle for $\Ad\rho$. 
We give $\mathcal C(\G,\sf G)$
the compact-open topology.
If $\rho$ is a smooth point of $\Hom(\pi_1(S),\sf G)$, then the space of cocycles for $\Ad\rho$ is canonically identified with $ \sf T_\rho\Hom(\pi_1(S),\sf G)$. In general,
the space of cocycles for $\Ad\rho$ is identified with 
Zariski tangent space of $\Hom(\pi_1(S),\sf G)$ at $\rho$ (see \cite{goldman-symplectic}).

\begin{corollary}\label{complexgroups} Let $\sf G$ be a complex simple Lie group with Lie algebra $\frak g$.
\begin{itemize}
\item[i)] The subset of $\mathcal C(\pi_1(S),\sf G)$ consisting of pairs $(\rho,\co)$ such that $F_{(\rho,\co)}(\pi_1(S))$ acts properly discontinuously on 
$\frak g$ has non-empty interior.
\item[ii)] The set of representations of  $\pi_1(S)$ into $\sf G\times\sf G$ that are discrete and faithful in each factor, and whose action on 
$\sf G$ via left/right multiplication is proper, has non-empty interior.
Moreover,  its closure contains product representations $\rho\times\rho$ where $\rho$ is totally Anosov.
\end{itemize}
\end{corollary}

The key tool in the proof of (i)  is the fact that the normalized
Margulis spectrum varies continuously (see Sambarino \cite[Lemma 2.34]{sambarino}). 
The proof of (ii) makes crucial use of a properness
criterion of Benoist \cite{BenoistK-criterion} and Kobayashi \cite{BKobayashi-criterion}.
Danciger, Gu\'eritaud and Kassel \cite[Prop. 1.8]{DGK-aff} previously established part (ii) of Corollary \ref{complexgroups} in the case
when $\sf G=\PSL(2,\C)$.

\medskip\noindent
{\bf Historical background:}
Auslander \cite{auslander} conjectured that if $\Gamma$ acts properly discontinuously and cocompactly on $\R^n$ by affine transformations,
then $\Gamma$ is virtually solvable. The conjecture remains open, but has been established when $n=3$ by Fried-Goldman \cite{fried-goldman},
when $n=4 $ and $n=5$ by Tomanov \cite{tomanov} and when $n=6$  by  Abels, Margulis, and Soifer \cite{AMS}.
Milnor \cite{milnor} asked whether, in analogy with the Bieberbach theorems, something similar might be true for actions which are not cocompact.
Margulis \cite{Mar1,Mar2} produced the first examples of proper affine actions by non-abelian free groups on $\R^3$, which
are now called Margulis space-times.
Drumm \cite{Drumm}, Charette-Drumm-Goldman \cite{CDG-2gen},  and Drumm-Goldman \cite{DG-crooked} introduced a geometric viewpoint
on Margulis' construction and produced large classes of new examples. Goldman, Labourie and Margulis  \cite{GLM} gave 
an exact criterion for when affine actions of free groups on $\R^3$ are proper (see Ghosh-Treib \cite{ghosh-treib} for extensions to $\mathbb R^{2n+1}$).
Danciger, Gu\'eritaud and Kassel \cite{DGK1,DGK2}
gave a complete classification of Margulis space-times with convex cocompact linear part and showed that their quotients are all
homeomorphic to the interior of a handlebody.

Smilga \cite{Smig} extended  Goldman, Labourie and Margulis' properness criterion to any real semisimple  Lie group $\mathsf{G}$ of the non-compact type. He
constructed a proper affine  action of a non-abelian free group on the Lie algebra $\mathfrak g$ whose  linear part is  Zariski dense in
$\Ad(\mathsf{G})$. Burelle and Zager Korenjak \cite{BZ,neza} further analyzed higher-dimensional Margulis space times.

 It is natural to ask which other groups admit proper affine actions.
 Mess  \cite{Mess} showed that a closed surface group cannot have a proper affine action on $\R^3$ (see \cite{Mess}). 
 Danciger and Zhang  \cite{DZ} showed that a proper affine action of a surface group on $\mathbb R^d$ cannot have a linear part
 which is a Hitchin representation into $\mathsf{SL}(d,\mathbb R)$.
 (see also Labourie \cite{labourie-entropy}). In a major breakthrough,
 Danciger, Gu\'eritaud and Kassel \cite{DGK-aff} proved that any right-angled Coxeter group on $k$ generators admits a proper affine actions on 
 $\R^{k(k-1)/2}$.  They also show that every hyperbolic surface group admits a proper affine action on $\R^6$. 
 
 We briefly describe the  construction of  Danciger, Gu\'eritaud and Kassel in \cite{DGK-aff}  for producing proper 
 affine actions of hyperbolic surface groups  on $\mathfrak{sl}(2,\CC)= \R^6$ whose linear part is the adjoint action of the a quasifuchsian group.  
 A surface group of genus $g$ is an index 4 subgroup of the fuchsian Coxeter group which is the reflection group of a $(2g+2)$-gon all of whose 
 internal angles are $\frac{\pi}{2}$. They describe an explicit path in the deformation space of the Coxeter group within $\PSL(2,\C)$ and
 show that the nonfuchsian points on this path give rise to proper affine actions.
 
\medskip\noindent{\bf Acknowledgements:} The authors would like to thank Sourav Ghosh, Fanny Kassel and Curt McMullen for helpful
comments or conversations. 

\section{Background}

\subsection{Geodesic currents and measured laminations}
\label{currents}
Bonahon \cite{bonahon-currents} introduced the space  $\mathcal C(X)$ of geodesic currents  on a closed hyperbolic surface $X$
as a natural closure of the space of   weighted collections of  closed geodesics on $X$.
Let $G(\Hp)$ be the space of unoriented geodesics in $\Hp$.  
If $X = \Hp/\Gamma$ is a hyperbolic surface, a {\em geodesic current} on $X$ is a locally finite $\Gamma$-invariant measure on $G(\Hp)$. 
We endow the space $\mathcal C(X)$ of geodesic currents on $X$  with the weak$^*$ topology.
Any closed geodesic $\alpha$ (with weight one) on $X$ gives rise to a geodesic current by considering the measure which has an
atom of weight one on each geodesic in $\mathbb H^2$ which covers $\alpha$.

Bonahon \cite{bonahon-currents} showed that one may continuously extend the length function on geodesics to a length
function
$$L_X:\mathcal C(X)\to [0,\infty).$$
We give a brief description of his construction.
Let $\sf{PT}(\Hp)$ denote the projective tangent bundle of $\Hp$ and consider  the  trivial fiber bundle 
$\pi:\sf{PT}(\Hp) \rightarrow G(\Hp)$ mapping a project  tangent vector a point  to the associated unoriented geodesic through the point.
The fiber above a geodesic $g$ can be identified with $g$ and given its length measure $\ell$. So, if $\mu\in \mathcal C(X)$, then we
obtain a $\Gamma$-invariant measure on $\sf{PT}(\Hp)$ which is locally $\mu\times\ell$. This measure
descends to a measure $\widehat{\mu\times\ell}$ on $\sf{PT}(X)$. Bonahon defines $L_X(\mu)$ to be the total mass of
$\widehat{\mu\times\ell}$.

Given homotopy classes  $\alpha$ and $\beta$ of closed geodesics on $X$, one defines their
{\em geometric intersection number} $i(\alpha,\beta)$ to be the number of transverse intersection points of $\alpha$ and $\beta$.
Bonahon also showed that one may continuously extend the intersection number to a continuous function on $\mathcal C(X)\times\mathcal C(X)$.
If $\mu,\nu\in \mathcal C(X)$, then
then $\mu\times\nu$ gives a  measure on $G(\Hp)\times G(\Hp)$ invariant under the diagonal action of $\Gamma$. Letting 
$$D(\Hp) = \{(\vec u,\vec v) \in \sf{PT}_x(\Hp)\times \sf{PT}_x(\Hp)\ |\ x\in\mathbb H^2\text{ and } \vec u \neq \vec v\}$$
we can identify $D(\Hp)$ with a subset of $G(\Hp)\times G(\Hp)$ by mapping $(\vec u,\vec v)$ to the pair of geodesics tangent to them. Then
$$D(X)  = D(\Hp)/\Gamma =  \{(\vec u,\vec v) \in \sf{PT}_x(X)\times \sf{PT}_x(X)\ | \ x\in X\text{ and } \vec u \neq \vec v\}.$$
Since the pullback of the  measure $\mu\times\nu$ to $D(\Hp)$ is invariant under the action of $\Gamma$ it projects to
a measure $\widehat{\mu\times\nu}$ on $D(X)$. Bonahon defines
$i_X(\mu,\nu)$  to be the total mass of $\widehat{\mu\times\nu}$ on $D(X)$.

If $S$ is a closed topological surface of genus at least $2$, we may define the space $\mathcal C(S)$ as the space of 
$\pi_1(S)$-invariant locally finite measures on the space of unordered pairs of distinct points in the Gromov boundary
$\partial_\infty\pi_1(S)$ of $\pi_1(S)$. Given $(X,h)\in\mathcal T(S)$, there is a well-defined
$h_*$-equivariant homeomorphism
$\xi:\partial_\infty:\pi_1(S)\to \partial \mathbb H^2$, which gives rise to a canonical identification of $\mathcal C(S)$
to $\mathcal C(X)$. We notice that $L_X$ does not induce a well-defined function on $\mathcal C(S)$, since the length
function depends on the underlying hyperbolic structure. However, the definition of intersection number does not
depend on the choice of hyperbolic structure so we have a well-defined continuous function
$$i:\mathcal C(S)\times\mathcal C(S)\to [0,\infty)$$
where for any $(X,h)\in\mathcal T(S)$, we take  $i(\mu,\nu)=i_X(\xi_*(\mu),\xi_*(\nu))$.

Notice that a closed geodesic $\alpha$ on $X$ is simple if and only if $i(\alpha,\alpha)=0$. More generally, we may consider the 
space  $\mathcal{ML}(X)$ of measured geodesic laminations on $X$ to be the set of geodesic currents of self-intersection zero.
If $\mu \in \mathcal{ML}(X)$ then its support $\supp(\mu)$ is a closed set foliated by disjoint complete geodesics, i.e. a
geodesic lamination.
Any interval $\alpha$ transverse to the support, inherits a measure by taking the $\mu$-measure of all the geodesics transverse to $\alpha$. 
This gives the {\em transverse measure} associated to $\mu$. Since intersection number is independent of the hyperbolic structure,
$\mathcal{ML}(S)$ is also well-defined.

\subsection{The convex core and bending laminations}
If $\rho$ is  quasifuchsian, but not fuchsian, there exists an orientation-preserving  homeomorphism 
$$j_\rho:S\times [0,1]\to C(\rho)$$ in the homotopy class
of $\rho$. Thurston \cite{thurston-notes} showed that the intrinsic metric on each component of $\partial C(\rho)$ is hyperbolic. 
Moreover, there exist measured laminations $\beta_+$ and $\beta_-$ on $\partial C_+(\rho)=j_\rho(S\times \{1\})$ and 
$\partial C_-(\rho)=j_\rho(S\times \{0\})$ so that $\partial C_\pm(\rho)$ is totally geodesic on the complement of  the support $\supp(\beta_\pm)$ 
of $\beta_\pm$ and
each leaf of $\supp(\beta_\pm)$ is mapped to a geodesic. (See Epstein-Marden \cite{EM} for complete proofs.)

We now briefly describe  the transverse bending measures  given by $\beta_\nu$ for $\nu \in \{\pm\}$. For a complete description see Epstein-Marden \cite{EM}. 
It will be easiest to work in the universal cover. Let $\tilde\partial C_\nu(\rho)$ denote the pre-image of $\partial C_\nu(\rho)$ in $\partial CH(\rho)$.
A {\em support half-space} to $\tilde\partial C_\nu(\rho)$ at $x \in \partial C_\nu(\rho)$ is a half-space $H$ such that  the interior of $H$ 
is disjoint from the convex hull $CH(\rho)$ and $x\in\partial H$. The boundary $\partial H$ of a support half-plane, is called a {\em support plane}.
If $\alpha:[0,1]\rightarrow \tilde\partial C_\nu(\rho)$ is an arc  transverse to $\beta_\nu$ (so that $\alpha(0),\alpha(1)\notin\supp(\beta_\nu)$), 
$\mathcal P=\{0 = t_0 < t_1 <\ldots < t_n = 1\}$ 
is a partition of $[0,1]$ and $\mathcal H=\{H_{t_i}\}$ is a collection of support planes to
$\tilde\partial C_\nu(\rho)$ at $\alpha(t_i)$, we define
$$i_{\mathcal P,\mathcal H}(\alpha, \beta_\nu) = \sum_{i=0}^{n-1}\theta_{i}$$
where $\theta_i$ is the angle between $H_{t_i}$ and $H_{t_{i+1}}$.
For the definition we  need to restrict  to pairs $(\mathcal P,\mathcal H)$ with the property that $H_{t_i}$ intersects $H_{t_{i+1}}$. 
Furthermore it is natural to restrict to pairs $(\mathcal P,\mathcal H)$ such that if $t \in [t_i, t_{i+1}]$ and $H_t$ is a support half-space to $\alpha(t)$ then $H_t$ intersects both $H_{t_i}$ and $H_{t_{i+1}}$.

We define the bending measure, as the limit over any sequence $\{(\mathcal P_n,\mathcal H_n)\}$ so that the mesh size of $P_n$ tends to zero.  
Since the intersection number $i_{\mathcal P,\mathcal H}(\alpha, \beta_\nu)$ is monotonically decreasing under refining the partition, 
we can also define the bending measure for $\alpha$ to be 
$$i(\alpha, \beta_\nu)  = \inf_{\mathcal P, \mathcal H}i_{\mathcal P, \mathcal H}(\alpha, \beta_\nu).$$

\subsection{Complex length}
If $g$ and $h$ are oriented geodesics in $\mathbb H^3$ which do not share an endpoint, then one can define a 
(unsigned) {\em complex distance} $\sigma(g,h)$ between them. The real part of $\sigma(g,h)$ is simply the distance between
$g$ and $h$, the complex part is the angle between $g$ and the parallel translate of $h$ along the unique common
perpendicular joining $g$ to $h$ measured counterclockwise in the plane spanned by $g$ and the parallel translate of $h$. 
The unsigned complex distance naturally lies in $\mathbb C/2\pi i \mathbb{Z}$ and 
varies smoothly over the space of pairs of oriented geodesics.

Kouroniotis \cite[Lemma 2.1]{kouron} shows that the unsigned complex distance is determined by the relation
 $$\cosh\sigma(g,h) = \frac{[g_-, h_-, h_+,g_+]+ 1}{[g_-, h_-, h_+,g_+]-1}$$
 where $[\cdot,\cdot,\cdot,\cdot]$ is the  cross ratio on the Riemann sphere given by
 $$[u,p,q,v] = \frac{(u-q)(v-p)}{(u-p)(v-q)}.$$
 It follows, in particular, that $\cosh\sigma$ is a well-defined smooth complex valued function.  As this is the only function that appears in our calculations, 
 we don't need to consider evaluating $\sigma$ directly. See Series \cite{series} for a helpful discussion of unsigned complex distance. 
 
Notice that if  $(h_-, h_+)  = (0,\infty)$, and $(g_-,g_+)= (u,v)$, then
$$[g_-, h_-, h_+,g_+] = \frac{v}{u} \quad\text{and}\quad \cosh\sigma(g,h) = \frac{v+u}{v-u}.$$
The following formula for the imaginary part of $ \cosh\sigma(g,h)$ will be used later:
\begin{equation}
\label{Im-cr} \Im \Big(\cosh\sigma(g,h)\Big)= -\frac{2\Im [g_-, h_-, h_+,g_+]}{|[g_-, h_-, h_+,g_+]-1|^2}.
\end{equation}

One can define the {\em complex length} $\mathcal L(\gamma)$ of a hyperbolic isometry $\gamma$ of $\mathbb H^3$. The real part
of $\mathcal L(\gamma)$ is the translation length of $\gamma$ along its axis, while its imaginary part is its
rotational angle. So, the complex length $\mathcal L(\gamma)$ is only defined modulo $2\pi i \mathbb Z$.

If $z\in\mathbb C$ and $A(z)$ is the hyperbolic transformation $w\rightarrow e^z w$, then $\mathcal L(A(z))=z$. 
If $g$ is an oriented  geodesic in $\mathbb H^3$ and $t\in\mathbb C$, we let $Q$ be an isometry mapping the pair $(0,\infty)$ to the pair $(g_-,g_+)$
and set $R(g,z) = Q A(z) Q^{-1}$. Notice that $R(g,z)$ is well-defined, varies holomorphically in $z$ and  $\mathcal L(R(g,z))=z$.

\subsection{Bending deformations}

Kourouniotis \cite{kouron-cont} studied the behavior of bending deformation using the theory of complex measured laminations.
If $X$ is a hyperbolic structure on a closed surface, 
a {\em complex measured lamination} $\mu$ on $X$ is a complex-valued, locally finite measure on
$G(\Hp)$ whose support is a geodesic lamination.
The space $\mathcal{ML}_\CC(X)$ of complex measured lamination is topologized with  the weak$^*$ topology on measures. 
A complex measure lamination $\mu$ induces a complex measure on arcs transverse to $\supp(\mu)$.

We will be interested in the subspace $\mathcal{ML}^{++}(X)$ of $\mathcal{ML}_\CC(X)$ consisting of measures
which take values with non-negative real and imaginary parts. 
Notice also that $\mathcal{ML}^{++}(S)$ is well-defined.

We first define the bending deformation of a quasifuchsian representation $\rho$ along $\mu\in \mathcal{ML}^{++}(S)$ when the support of $\mu$ has finitely many leaves.
Since $X$ is a hyperbolic structure on $S$, $X=\mathbb H^2/\Gamma$ and this gives an identification of
$\partial_\infty\pi_1(S)$ with $\partial\mathbb H^2$. Consider the $\rho$-equivariant map $\xi_\rho:\partial \mathbb H^2\to\partial\mathbb H^3$.
The map $\xi$ induces a $\rho$-equivariant map $\xi_\rho^*:\mathcal G(\Hp)\to\mathcal G(\mathbb H^3)$ where $\mathcal G(\mathbb H^n)$ is the
space of oriented geodesics in $\mathbb H^n$. We fix a basepoint $x_0\in X$, in the complement of $\supp(\mu)$ and a preimage 
$\tilde x_0$ of $x_0$ in $\mathbb H^2$.
Given $\gamma\in \pi_1(X,x_0)$, 
let $m_1,\ldots, m_n$ be the geodesics in the support of $\tilde \mu$ intersecting $ [\tilde x_0, \gamma(\tilde x_0)]$ with  
atomic measures $a_1,\ldots a_m$.  We order the leaves from right to left and orient each $m_i$ so that they cross
$ [\tilde x_0, \gamma(\tilde x_0)]$ from left to right.
We then define, for all $z\in\mathbb C$,
$$\rho_{z\mu}(\gamma) =  R(\xi_*(m_1),za_1)R(\xi_*(m_2),za_2)\ldots R(\xi_*(m_{n-1},za_{m-1})R(\xi_*(m_n),za_n)\rho(\gamma).$$
Kouroniotis observes that $\rho_{z\mu}$ is a representation for all $z$ and is quasifuchsian  for $z$ lying in
an open neighborhood of $0$ in $\C$. 

Kouroniotis showed that if $\{\mu_n\}$ is a sequence of finite leaved complex measured laminations converging to $\mu\in \mathcal{ML}_\CC(X)$,
then $B_{z\mu_n}$ converges to $B_{z\mu}$ and we obtain a map
$$B_\mu: \QF(S)\times \CC\to \frak X(S)\ \ 
\text{ given by }\ \ 
B_\mu([\rho],z) = [\rho_{z\mu}]$$
which is holomorphic in $z$ (see \cite{kouron-cont}). 
The associated holomorphic bending vector field $T_\mu: \QF(S) \rightarrow T(\QF(S))$ is defined to be
$$T_\mu([\rho]) = \left.\frac{\partial }{\partial z}\right|_{z=0} B_\mu([\rho],z)$$
 
Kourouniotis further proved

\begin{theorem}{\rm (Kourouniotis \cite[Theorem  3]{kouron-cont})}
\label{continuity of derivative}
 The map $T:\mathcal{ML}^{++}(S)\times \QF(S)\rightarrow T(\QF(S))$ given by $T(\mu,[\rho]) = T_\mu([\rho])$ is continuous and  holomorphic in $[\rho]$. 
\label{kourounitis-cont}
\end{theorem}

 \section{Variation of complex length}
 
In this section, we obtain our formula for the variation of the complex length function during a bending defornation along a measured
lamination $\mu$. Kouroniotis \cite{kouron} already obtained a formula 
when the support of $\mu$ has finitely many leaves.

\begin{theorem}{\rm (Kouroniotis \cite[Thm 4.1]{kouron})}
\label{finite derivative of complex length}
If $\rho$ is quasifuchsian, $\mu\in \mathcal{ML}(S)$ has finite-leaved support, $S\cong X=\mathbb H^2/\Gamma$,
$\gamma\in \pi_1(X,x_0)$, $x_0\notin\supp(\mu)$ and
$\mathcal L_\gamma(z)=\mathcal L(\rho_{z\mu}(\gamma))$, then
$$\mathcal L_\gamma'(0) = \sum_{j=1}^n a_j\cosh(\sigma((\xi_\rho)_*(a(\gamma)),(\xi_\rho)_*(m_j))$$
where $[\tilde x_0,\gamma(\tilde x_0)]$ intersects $\tilde\mu$ in the leaves $\{m_1,\ldots,m_n\}$ (ordered from right to left and oriented so they
cross $[\tilde x_0,\gamma(\tilde x_0)]$ from left to right) with atomic weights
$\{a_1,\ldots,a_n\}$ and $a(\gamma)$ is the axis of $\gamma$.
\end{theorem}

We will use Theorem \ref{continuity of derivative} to extend Kourouniotis' formula to the general case.
We first re-interpret Kourouniotis' formula in a more measure-theoretic fashion. If
$\gamma\in \pi_1(X,x_0)$ is non-trivial, we let 
$$D_\gamma(X)=\{(\vec u,\vec v)\in D(X) : x\in\gamma^*,\ \vec u \text{ is tangent to } \gamma^*\text{ at } x, \text{ and } \vec u\ne\vec v\}.$$
If $(\vec u,\vec v)\in D_\gamma(X)$, we lift $\vec u$ and $\vec v$ to a pair $(\vec U,\vec V) \in \sf{PT}_{\tilde x}(\Hp)\times \sf{PT}_{\tilde x}(\Hp)$
with  $ \vec U$ tangent to the axis $a_U(\gamma)$ of $\gamma$ at  a lift $\tilde x$ of $x$. 
Then let $g_V$ be an oriented geodesic in $\mathbb H^2$ which has $\vec V$ as a tangent
vector and crosses $a_U(\gamma)$ from left to right.
We define a continuous function
$$\sigma_{\rho,\gamma}:D_\gamma(X)\to \mathbb C/2\pi i \mathbb Z$$
given by 
$$\sigma_{\rho,\gamma}(\vec u,\vec v)=\sigma\Big((\xi_\rho)_*(a_U(\gamma)),(\xi_\rho)_*(g_V)\Big).$$
If we regard $\gamma$ as a measured geodesic lamination with atomic measure, then the support of
the measure $\widehat{\gamma\times\mu}$ on $D(X)$ lies within $D_\gamma(X)$, so we may
view $\widehat{\gamma\times\mu}$ as a measure on $D_\gamma(X)$. Then notice that 
$$\mathcal L_\gamma'(0)=\int_{D_\gamma(X)} \cosh \sigma_{\rho,\gamma}\ d\widehat{(\gamma\times\mu)}.$$
Notice that this formula does not depend on the choice of $X$.

We use Kourouniotis' later work  \cite{kouron-cont} on the continuity of bending to extend this formula  to the general case.

\begin{theorem}
\label{derivative of complex length}
If $\rho$ is quasifuchsian, $\mu\in \mathcal{ML}^{++}(S)$, $S\cong X = \mathbb H^2/\Gamma$,  $\gamma\in \pi_1(X,x_0)$ and
$\mathcal L_\gamma(z)=\mathcal L(\rho_{z\mu}(\gamma))$, then
$$\mathcal L_\gamma'(0) = \int_{D_\gamma(X)}  \cosh \sigma_{\rho,\gamma}\ d\widehat{(\gamma\times\mu)}. $$ 
\end{theorem}

\begin{proof}
First notice that our formula holds trivially when $i(\gamma,\mu)=0$, since both sides are clearly equal to 0.
So we will assume that $i(\gamma,\mu)\ne 0$.

Fix $(X,h)\in\mathcal T(S)$ and
let $\{\mu_n\}$ be a sequence of finite-leaved complex measured laminations in $\mathcal{ML}^{++}(X)$ which converge to $\mu$.
We claim that there exists $\theta>0$  and $N$ so that if $n\ge N$ and $(\vec u,\vec v)$ is in the support of
$\widehat{\gamma\times\mu_n}$, then the (un-oriented) angle $\angle \vec u,\vec v$ between $\vec u$ and $\vec v$ lies between
$\theta$ and $\pi-\theta$.
If not, then $\gamma$ lies in the Hausdorff limit of a subsequence of $\{\supp(\mu_n)\}$. Since the support of $\mu$ is contained
in the Hausdorff limit of any subsequence of $\{\supp(\mu_n)\}$, this would imply that $i(\gamma,\mu)=0$.

Let $\mathcal L_\gamma^n(z)=\mathcal L(\rho_{z\mu_n}(\gamma))$ and $\mathcal L_\gamma(z)= \mathcal L(\rho_{z\mu}(\gamma))$. Each 
$\mathcal L_\gamma^n$ is  holomorphic in $z$ (by Theorem \ref{continuity of derivative}) and  
$\{\mathcal L_\gamma^n\}$ converges uniformly on compact sets to $\mathcal L_\gamma$. Thus,
$$\mathcal L_\gamma'(0)=\lim (\mathcal L_\gamma^n)'(0).$$
Moreover, Theorem \ref{finite derivative of complex length} implies that
$$(\mathcal L_\gamma^n)'(0)= \int_{D_\gamma(X)} \cosh\sigma_{\sigma,\rho}\ d\widehat{(\gamma\times\mu_n)}$$
for all $n$.
Let
$$D_{\gamma,\theta}(X)=\{(\vec u,\vec v)\in D_\gamma(X) : \pi-\theta\ge\angle \vec u,\vec v\,\ge\theta\}.$$
For all $n\ge N$,
$$(\mathcal L_\gamma^n)'(0)= \int_{D_{\gamma,\theta}(X)} \cosh\sigma_{\sigma,\rho}\ d\widehat{(\gamma\times\mu_n)}$$
Since, $D_{\gamma,\theta}(X)$ is compact and $\{\gamma\times\mu_n\}$ converges
to $\gamma\times\mu$ in the weak$^*$-topology, we see that
$$\lim_{n\rightarrow \infty} (\mathcal L_\gamma^n)'(0)= \int_{D_{\gamma,\theta}(X)} \cosh\sigma_{\sigma,\rho}\ d\widehat{(\gamma\times\mu)}=
 \int_{D_{\gamma}(X)} \cosh\sigma_{\sigma,\rho}\ d\widehat{(\gamma\times\mu)}$$
which completes the proof.
\end{proof}

\section{Critical points of the entropy function}
\label{entropy}

In this section we show that if one of the components of the domain of discontinuity of a nonfuchsian quasifuchsian group
is moderately bent, then it is not a critical point of the entropy function

Work of Sambarino \cite{sambarino} provides a criterion guaranteeing that $\rho$ is not a critical point.
If we consider $[\rho]\in \QF(S)$ and a non-zero variation $\vec v\in T_{[\rho]} \QF(S)$,
then the link between variation of entropy and variation of lengths is given by the \emph{set of normalized variations}
$$\mathbb V_{\vec v}=\overline{\left\{\frac{\dd\ell_\g(\vec v)}{\ell_\g}:\g\in\pi_1(S)\right\}}\subset\R.$$ 
This is a closed interval and if $\rho$ is not fuchsian  it has non-empty interior. 
Moreover, in this case $-\dd h(\vec v)/h(\rho) \in\mathrm{int}(\mathbb V_{\vec v})$ (see \cite[Lemma 2.32]{sambarino}).  
The following criterion follows immediately.

\begin{proposition}{\rm (Sambarino \cite{sambarino})}
\label{critical point fact} 
Suppose that $[\rho]\in \QF(S)$ is not fuchsian and $\vec v\in T_{[\rho]} \QF(S)$ is not zero. 
If $\dd\ell_\g(\vec v)\leq0$ for all $\g\in\pi_1(S)$, then $\dd h(\vec v)\neq0,$ in particular $h$ is not critical at $[\rho].$
\end{proposition}

We can use our formula for the derivative of complex length and our angle bounds to produce such vectors.

\medskip\noindent
{\bf Theorem \ref{Main Theorem 1}.} {\em
Suppose that $[\rho]\in QF(S)$ is not fuchsian, $\Omega_\nu(\rho)$ is moderately bent (for some $\nu\in\{\pm\}$) and
$w=w_\nu(\rho)$ is the infinitesmal bending deformation for $-i\beta_\nu$.
If $\gamma\in \pi_1(S)$, then $d\ell_\gamma(w)\le 0$ and $d\ell_\gamma(w)  < 0$ if  $\gamma$ intersects $\beta_\nu$ essentially.}

\medskip

Notice that Theorem \ref{Main Theorem 2} follows immediately from Theorem \ref{Main Theorem 1} and Proposition \ref{critical point fact}.

\begin{proof}
By  Theorem \ref{derivative of complex length},
$$d\ell_\gamma(w)=
\Re(d\mathcal L_\gamma(w))= \Re\left( \int_{D_\gamma(X)}  \cosh \sigma_{\rho,\gamma}\ d\widehat{(\gamma\times -i\beta_\nu)}\right) = 
\int_{D_\gamma(X)}  \Im\Big(\cosh \sigma_{\rho,\gamma}\Big)\ d\widehat{(\gamma\times\beta_\nu)}.$$
In particular, $d\ell_\gamma(w)=0$ if $\gamma$  does not intersect $\beta_\nu$ essentially.
Recall that given $(\vec u, \vec v) \in D_\gamma(X)$, we lift $(\vec u, \vec v)$ to a pair of projective tangent vectors $(\vec U,\vec V)\in \sf{PT}\Hp$
which are  tangent to geodesics $a_U(\gamma)$ and  $g_V$. Then 
$$\sigma_{\rho,\gamma}(\vec u,\vec v)\ =\sigma((\xi_{\rho})_*(a_U(\gamma)),(\xi_{\rho})_*(g_V)).$$
Denoting the geodesics as an oriented pair of endpoints, we let $(u_-,u_+) = (\xi_{\rho})_*(a_U(\gamma))$ and $(v_-,v_+) =  (\xi_{\rho})_*(g_V)$. 
By Equation (\ref{Im-cr})
$$\Im\Big(\cosh \sigma_{\rho,\gamma}\Big)(\vec u,\vec v)\ =\frac{-2\Im\Big( [u_-,v_-,v_+,u_+]\Big) }{\left| [u_-,v_-,v_+,u_+]-1\right|^2}.$$
Thus it suffices to show that $\Im\Big([u_-,v_-,v_+,u_+]\Big) > 0$ for all $(\vec u,\vec v)\in\supp (\widehat{\gamma\times\beta_\nu})$.
Notice that, since $g_V$ crosses $a_U(\gamma)$ from left to right, the vertices occur in the  order $u_-,v_-,u_+,v_+$ on $\Lambda(\rho)$.

For a bending line with endpoints  the bending pair $(v_-,v_+)$, we can choose a support plane $P$ which intersects $\Lambda(\rho)$ only at $v_-$ and $ v_+$.
(If the bending line has a unique support plane it has this property, while if not we can take any support plane which is not left-most or right-most.)
 We may assume that  $(v_-,v_+) =(0,\infty)$  that the plane $P$ lying above the real axis $\R$ is a support plane for $CH(\Lambda(\rho))$, 
 and the upper half-plane is the boundary disk of the support half-space bounded by $P$. 
 It follows that  $\Omega_\nu$ contains the upper half-plane and that $u_-$ and $ u_+$ are in the lower half-plane. 
Since $\Omega_\nu$ is moderately bent, there exists a line $L$ through $0$ and $\infty$ which intersects $\Lambda(\rho)$ 
transversely and only at $0$ and $\infty$.  With our conventions $u_-$ lies to the left of $L$ (i.e. the side containing the negative real axis) 
and $u_+$ lies to the right of $L$ (see Figure \ref{fig:mod_bend}).
As the cross-ratio is invariant under rotation, we apply a clockwise rotation $r$ about 0 so that $z_+ = r(u_+)$ lies on the positive real axis and
let $z_- = r(u_-)$. The angle of rotation is less than the  angle the portion of  $L$ below the real axis  makes with the positive real axis, so $\Im(z_-) < 0$.  
\begin{figure}[htbp] 
   \centering
   \includegraphics[width=3in]{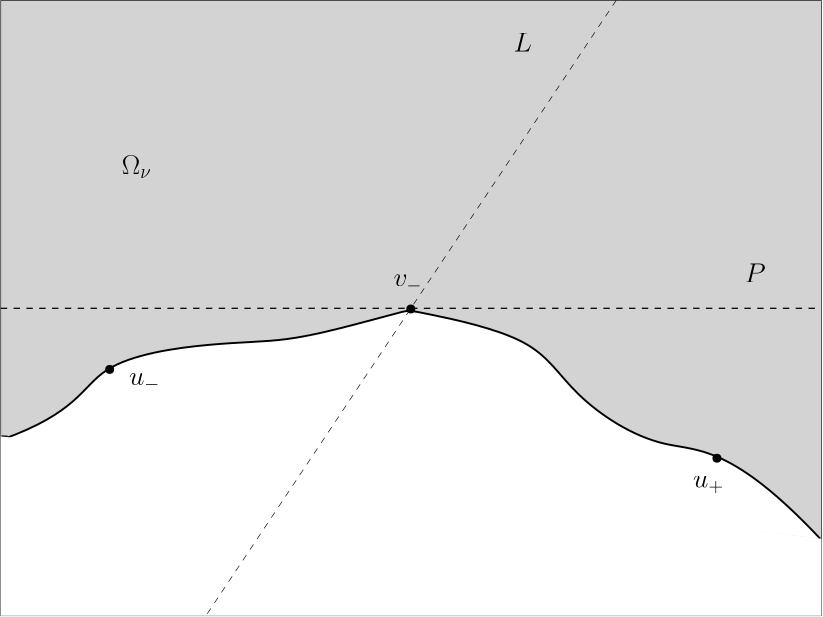} 
   \caption{Domain $\Omega_\nu$}
   \label{fig:mod_bend}
\end{figure}
So, 
$$ [u_-,v_-,v_+,u_+] = [z_-,0,\infty,z_+]=\frac{z_+}{z_-}=\frac{z_+\overline{z_-}}{|z_-|^2}.$$
Moreover,
$$\Im\Big(z_+\overline{z_-}\Big)= - z_+\Im(z_-) > 0,\quad\text{so}\quad \Im\Big( [u_-,v_-,v_+,u_+]\Big) > 0,$$
which completes the proof.
\end{proof}

\section{Binding pairs of laminations and $\delta$-intersection number}

In this section we prove Theorem \ref{Main Theorem 4} which shows that lengths of curves decrease uniformly
in the direction  $w(\rho)= w_+(\rho)+w_-(\rho)\in \sf T_{[\rho]}QF(S)$. 

We say that a pair $(\alpha,\beta)$ of measured laminations  on $S$ is {\em binding} if 
$$i(\alpha,\mu)+i(\beta,\mu)>0$$
whenever $\mu$ is a non-trivial geodesic current on $S$.
If $(\vec u,\vec v)\in D(X)$, let $\angle \vec u, \vec v$ is the unoriented angle between $\vec u$ and $\vec v$.
In particular, $(\vec u,\vec v)\in (0,\pi))$. If $\delta\in \left(0,\frac{\pi}{2}\right)$, we define
$$D_\delta(X)=\{(\vec u,\vec v)\in D(X) : \pi-\delta\ge\, \angle \vec u, \vec v\, \ge\delta\}$$
If $\alpha,\mu\in \mathcal C(X)$ and $\delta>0$, we define $i_\delta(\alpha,\mu)$ to be the total mass of the points in 
$D_\delta(X)$ with respect to the quotient measure of $\alpha\times\mu$. Notice that while intersection number is independent of
the base hyperbolic structure, the $\delta$-intersection number depends crucially on the underlying structure.

\begin{lemma}
\label{small enough delta}
Let  $(\alpha,\beta)$ be a binding pair of measured laminations on a closed hyperbolic surface $X$.
There exists a$\delta > 0$  and $C>0$ so that if $\mu\in\mathcal C(X)$, then
$$i_\delta(\alpha,\mu)+i_\delta(\beta,\mu) \geq C \ell_X(\mu).$$
\end{lemma} 

\begin{proof}
Suppose no such $\delta$ and $C$ exist. Then, since $\ell_X$ and intersection number both scale linearly,
there exist a sequence $\{\mu_n\}$ of unit length geodesic currents so that
$$i_{\frac{1}{n}}(\alpha,\mu_n)+i_{\frac{1}{n}}(\beta,\mu_n) \to 0.$$
Since the space  $\mathcal C_1(X)$ of unit length geodesic currents  is homeomorphic to the space of projective geodesic currents $\mathcal{PC}(S)$,  $\mathcal C_1(X)$ is compact (see \cite{bonahon-currents}). We may assume that
$\mu_n\to\mu\in\mathcal C_1(X)$. Since $(\alpha,\beta)$ is binding,
$$i(\alpha,\mu)+i(\beta,\mu)>0.$$
Since $D(X)=\bigcup_{n\in\mathbb N} D_{\frac{1}{n}}(X)$, we see
that there exists $N\in\mathbb N$ so that  
$$i_{\frac{1}{N}}(\alpha,\mu)+i_{\frac{1}{N}}(\beta,\mu) > 0.$$
By considering a continuous bump function supported on $D_{\frac{1}{2N}}(X)$ and equal to 1 on $D_{\frac{1}{N}}(X)$,
we see that 
$$\liminf_{n\rightarrow \infty} \left(i_{\frac{1}{2N}}(\alpha,\mu_n)+i_{\frac{1}{2N}}(\beta,\mu_n)\right) \ge i_{\frac{1}{N}}(\alpha,\mu)+i_{\frac{1}{N}}(\beta,\mu)>0$$ which produces a contradiction as
$$\liminf_{n\rightarrow \infty} \left(i_{\frac{1}{2N}}(\alpha,\mu_n)+i_{\frac{1}{2N}}(\beta,\mu_n)\right) \le \liminf_{n\rightarrow \infty}\left( i_{\frac{1}{n}}(\alpha,\mu_n)+i_{\frac{1}{n}}(\beta,\mu_n)\right) = 0.$$ 
\end{proof}

We can now return to our bounds on derivatives and observe that they can be made uniform
with respect to $\delta$-intersection number.

\begin{lemma}
\label{uniform variation}
Suppose that $[\rho]\in \QF(S)$ is not fuchsian,
$\Omega_\nu$ is moderately bent for some $\nu\in\{\pm\}$ and $w=w_\nu(\rho)\in \sf T_{[\rho]} QF(S)$.
If $\delta\in \left(0,\frac{\pi}{2}\right)$, then  there exists $D=D(\delta,\rho_0)$ so that 
$$\dd\ell_\gamma(w_\nu(\rho)) \leq - C i_\delta(\gamma,\beta_+)$$
for all $\gamma\in\pi_1(S)$.
\end{lemma}

\begin{proof}
Recall that  if $\mu\in\mathcal C(X)$, then
$i_\delta(\mu,\beta_\nu)=\mu\times \beta_\nu \big(D_\delta(X)\big)$ and that $D_\delta(X)$ is compact.
If $(u,v)\in D_\delta(X)$ lies in the support of $\mu\times \beta_\nu$, then $u$ is a tangent vector to a leaf $m$ of $\mu$ 
and $v$ is a tangent vector to a leaf $a$ of $\beta_\nu$. Since $\Omega_\nu$ is moderately bent,
we saw in the proof of Theorem \ref{Main Theorem 1} that  $\Im\cosh \sigma(a,m) < 0$. 
Since the intersection of the support of $\mu\times \beta_\nu$ with $D_\delta(X)$ is compact and complex length varies continuously,
there exists $C>0$ so that if $(u,v)\in D_\delta(X)$ lies in the support of $\mu\times \beta_\nu$, then
$\Im\left(\cosh \sigma(a,m)\right)\le -C$.
Our formula for $d\ell_\gamma(w_\nu(\rho))$, from Theorem \ref{derivative of complex length}, immediately gives that
$$d\ell_\gamma(w_\nu(\rho)) \leq -C i_\delta(\gamma,\beta_+)$$
for all $\gamma\in\pi_1(S)$.
\end{proof}

We are now ready to prove Theorem \ref{Main Theorem 4}.

\medskip\noindent
{\bf Theorem \ref{Main Theorem 4}.} {\em
Suppose that $\rho\in \QF(S)$ is not fuchsian and that $w = w_+(\rho)+w_-(\rho)$. If both $\Omega_+(\rho)$ and $\Omega_-(\rho)$ are
moderately bent, then there exists  $K > 0$ such that
$$\dd\ell_\g(w) \leq  -K \ell_\g(\rho)$$
for all $\gamma\in\pi_1(S)$.}

\begin{proof}
Recall that the bending laminations of a quasifuchsian group which is not fuchsian always bind 
(see, for example Bonahon-Otal \cite[Proposition 4]{BO}).
Fix a hyperbolic structure  $X$ on $S$.
Lemma \ref{small enough delta} implies that  there exists a $\delta > 0$  and $C>0$ so that if $\mu\in\mathcal C(X)$, then
$$i_\delta(\beta_+,\mu)+i_\delta(\beta_-,\mu) \geq C \ell_\mu(X).$$
Lemma \ref{uniform variation} implies that there exists $D>0$ so that   
$$\dd\ell_\g(w_{\pm}(\rho)) \leq - D i_\delta(\g,\beta_\pm)$$
for all $\g\in\pi_1(S)$ and $\nu\in\{\pm\}$. 
Therefore,
$$\dd\ell_\g(w) \leq - D \Big( i_\delta(\g,\beta_+)+i_\delta(\g,\beta_-)\Big)\leq -CD \ell_\g(X)$$
for all $\g\in\pi_1(S)$.
Since $\rho$ is quasifuchsian, there exists $B>1$ so that
$$\frac{\ell_\g(\rho)}{B}\le \ell_\g(X)\le B\ell_\g(\rho)$$
for all $\g\in \pi_1(S)$.
Therefore,
$$\dd\ell_\g(w) \leq  -\frac{CD}{B} \ell_\g(\rho)$$
for all $\g\in\pi_1(S)$.
\end{proof}

\section{Proper affine actions  with quasifuchsian holonomy}
\label{affine actions}

In this section, we show that if both components of the domain of discontinuity of  a nonfuchsian quasifuchsian
representation $\rho$ are moderately bent, then $\Ad\rho$
arises as the linear part of a proper affine action of the surface group on $\mathfrak{sl}(2,\mathbb C)$.

In Section \ref{margulis}, we introduce the Margulis invariant of a loxodromic element in $\mathsf{PSL}(d,\mathbb C)$ and an element of
the Lie algebra $\mathfrak{sl}(d,\mathbb C)$. We see that the Margulis invariant provides a bound on the displacement of the associated action
on $\mathrm{Aff}(\mathfrak{sl}(d,\mathbb C))$.
The reader who is only interested in the quasifuchsian case can always assume that $d=2$
which simplifes the discussion. 

In Section \ref{Anosovpreliminaries1}, we recall the theory of totally Anosov (a.k.a. Borel Anosov) representations into $\mathsf{PSL}(d,\mathbb C)$.
This is a natural setting for  our work, since  if a representation is totally Anosov then all infinite order elements have loxodromic image.
A representation into $\mathsf{PSL}(2,\mathbb C)$ is totally Anosov if and only if it is convex cocompact.
In Section \ref{criterion}, we develop a criterion guaranteeing that a totally Anosov representation into $\mathsf{PSL}(d,\mathbb C)$ and an associated cocycle with image in
$\mathfrak{sl}(d,\mathbb C)$ give rise to a proper affine action on $\mathrm{Aff}(\mathfrak{sl}(d,\mathbb C))$.
In Section \ref{conclusion} we use our criterion to prove  Theorem \ref{Main Theorem 3}. In section \ref{others}
we obtain generalizations of Theorem \ref{Main Theorem 3} in other complex Lie groups.

\subsection{The Margulis invariant}
\label{margulis}

The fact that we are working in $\PSL(d,\mathbb C)$ means that the eigenvalues of a matrix are only well-defined up to mulitplication by a $d^{\rm th}$ root of unity.
However, the moduli of the eigenvalues and their associated eigenspaces are still well-defined. We first develop language which makes this concrete.

If $g\in\PSL(d,\C)$, we may choose a lift $\tilde g\in \SL(d,\mathbb C)$. Let $\big(\lambda_i(\tilde g)\big)_{i=1}^d$ denote the (generalized) eigenvalues of $\tilde g$ ordered
so that $|\lambda_i(\tilde g)|\ge |\lambda_{i+1}(\tilde g)|$ for all $i$. Notice that the center  $Z_d$ of $\mathsf{SL}(d,\mathbb C)$ consists of matrices of the form $r \rm{I}$ 
where $r$ is a $d^{\rm th}$ root of unity. 
If $\tilde g'$ is another lift of $g$, then there exists $c\in Z_d$, so that $\tilde g=c\tilde g'$. In particular,  there exists a  $d^{\rm th}$ root of unity $r$,  so that
$\lambda_i(\tilde g)=r\lambda_i(\tilde g')$  for all $i$. So we may define the {\em Jordan projection} $\mu:\PSL(d,\mathbb C)\to \mathfrak a^+$ where
$$\mathfrak a^+=\left\{\mathrm{diag}(a_1,\ldots,a_d):\ a_1, \ldots, a_d\in\R,\  a_i\ge a_{i+1}\forall\ i\text{ and } \sum_{i=1}^d a_i=0\right\}\ \ \text{and}\ \ \mu(g)=(\log|\lambda_1(\tilde g)|,
\ldots,\log |\lambda_d(\tilde g)|).$$

We say that $g$ is loxodromic if any (hence every) lift  $\tilde g$ is loxodromic, i.e. is 
diagonalizable over $\mathbb C$ and its eigenvalues have distinct moduli. 
If $g$ is loxodromic, then there exists $\tilde \psi\in\PSL(d,\C)$ that conjugates $\tilde g$ to a diagonal matrix $\mathrm{diag}(b_1,\ldots,b_d)$
with decreasing in modulus entries, i.e.
$|b_i|>|b_{i+1}|$ for all $i$. Then $\tilde\psi$ projects to $\psi\in\PSL(d,\C)$ and
we will say that $\psi$ conjugates $g$ into standard form.

If $g$ is loxodromic, let $L_i(g)$ denote the one-dimensional eigenspace of $\tilde g$ with eigenvalue $\lambda_i(\tilde g)$.
The eigenspace $L_i(g)$ is well-defined since any two lifts differ by an element of $Z_d$ and $Z_d$ acts trivially on $\mathbb P(\C^d)$.
Let $E_i(g)$ be the $i$-dimensional space spanned by the first $i$ eigenspaces, i.e. $E_i(g)=\langle L_1(g),\ldots, L_i(g)\rangle$.
The flag $E(g)=\big(E_i(g)\big)_{i=1}^{d-1}$ is the attracting flag of $g$. 
Notice that if $\psi$ is an element conjugating $g$
into standard form and 
$$E^0=\{ \langle e_1\rangle, \langle e_1, e_2, \rangle, \ldots,\langle e_1,\ldots,e_{d-1}\rangle\}$$
then $E(g)=\psi(E^0)$.
The repelling flag of $g$ is given by
$F(g)=\big(F_i(g)\big)_{i=1}^{d-1}$ where $F_i(g)=\langle L_{d-i+1},\ldots, L_d(g)\rangle$. Notice that $E(g)=F(g^{-1})$.

The Lie algebra $\sl(d,\C)$ of $\PSL(d,\C)$ contains a
subalgebra $\h$ consisting of traceless diagonal matrices (over $\C$). More precisely,
$$\h=\big\{\diag(z_1,\ldots,z_d):\sum z_i=0\big\}.$$ 
The adjoint action $\Ad(g):\sl(d,\C)\to\sl(d,\C)$ is given $\Ad(g)(v)=\tilde g v \tilde g^{-1}$.
(Notice that again the choice of lift is irrelevant since any two lifts differ by an element of $Z_d$ and $\Ad(c)$ is the identity map for any  $c\in Z_d$.)
If $g$ is loxodromic, $\Ad(g)$ is also diagonalizable (although it may have eigenvalue coincidences).
If $\psi$ conjugates $g$ into standard form, then its $1$-eigenspace is given by
$$V^0(g):=\mathrm{Fix}\,(\Ad(g))=\Ad(\psi^{-1})\h.$$
We can consider the projection 
$$\pi^0_g:\sl(d,\C)\to V^0(g)$$ 
whose kernel $W(g)$ is $\Ad(g)$-invariant. More explicitly $W(g)$ is  the sum of the eigenspaces associated to all the eigenvalues but $1$.

Margulis \cite{Mar1} originally defined the  Margulis invariant for pairs in $\mathrm{SO}(1,2)\ltimes\R^3$. The following
generalization is due to Smilga \cite{Smig}.

\begin{definition}  If $g\in\PSL(d,\C)$ and $x\in\sl(d,\C)$ the \emph{Margulis invariant} of $(g,x)$ is 
$$\sf m(g,x):=\Ad(\psi)\big(\pi^0_g(x)\big)\in\h,$$ 
i.e. we project $x$ onto the fixed set  $V^0(g)$ of $\Ad(g)$ and move it into $\h$ via $\Ad(\psi).$
\end{definition}

Since  any two elements $\psi$ and $\psi'$   which conjugate $g$ into standard form  satisfy $\psi'\psi^{-1}\in\exp\h$ and $\Ad(\exp \h)$ acts trivially on $\h,$ 
the Margulis invariant $\sf m(g,x)$ is independent  of the choice of $\psi$. Moreover,
\begin{itemize}
\item[-] $\sf m(g,x)$ is invariant under conjugation by an element of $\PSL(d,\C)\ltimes\sl(d,\C),$ and
\item[-] $\sf m\big((g,x)^{-1}\big)=\mathrm{i}\big(\sf m(g,x)\big),$ where $\mathrm{i}$ is the opposition involution of $\h.$
\end{itemize}

\medskip

In our proof it will be useful to choose $\psi$ efficiently.

\begin{lemma}{\rm (Bochi-Potrie-Sambarino \cite[Proposition 7.2]{BPS})} 
\label{control-psi} 
Given $d\in\mathbb N$, 
there exists $K>1$  so that 
if $g\in\PSL(d,\C)$ is loxodromic, one may choose an element $\psi$ which conjugates $g$ into standard form so that
$$\|\psi\|\min_{i\in\lb1,d-1\rb}\angle(E_i(g),F_{d-i}(g))\in(1/K,K)$$ 
where $\|\psi\|$ is the standard operator norm of $\psi$.
\end{lemma}

\begin{remark}[When $d=2$] If $g\in \PSL(2,\C)$ is loxodromic, then there exists $\psi\in \PSL(2,\C)$ so that
$\psi g\psi^{-1}=\pm\begin{bmatrix} \lambda & 0\\ 0 & \lambda^{-1}\end{bmatrix}$ where $|\lambda|>1$. So,
$$V^0(g):=\mathrm{Fix}\,(\Ad(g))=\Ad(\psi^{-1})\h=\left\langle\psi^{-1}\begin{bmatrix} 1 & 0\\ 0 & {-1}\end{bmatrix}\psi\right\rangle.$$
The other two eigenspaces of $\Ad(g)$ are 
$$\left\langle\psi^{-1}\begin{bmatrix} 0 & 1\\ 0 & 0\end{bmatrix}\psi\right\rangle\quad\text{and}\quad
\left\langle\psi^{-1}\begin{bmatrix} 0 & 0\\ 1 & 0\end{bmatrix}\psi\right\rangle$$
with eigenvalues $\lambda^2$ and $\lambda^{-2}$ respectively. Moreover, $E(g)$ is the attracting eigenline of $g$
and $F(g)$ is the repelling eigenline. Notice that the opposition involution is the identity map when $d=2$, so $\sf m\big((g,x)^{-1}\big)=\sf m(g,x).$
\end{remark}

An element $(g,x)\in\PSL(d,\C)\ltimes\sl(d,\C)$ induces an affine transformation 
$$\aff{(g,x)}:\sl(d,\C)\to\sl(d,\C)\quad\text{by}\quad \aff{(g,x)}v=\Ad(g)v+x.$$
When the angles between the attracting/repelling flags of $g$ are uniformly bounded below, then the displacement of $\aff{(g,x)}$ is controlled by the Margulis invariant:

\begin{lemma}[Margulis invariant as least displacement]
\label{minimo}  Given $C>0$  and $d\in \mathbb N$,  there exists $c>0$ such that if $g\in\PSL(d,\C)$ is loxodromic
and 
$$\min_{i\in\lb1,d-1\rb}\angle(E_i(g),F_{d-i}(g))\ge C,$$
then 
$$\big\|\aff{(g,x)}v-v\big\|\geq c\|\sf m(g,x)\|.$$
for any $x,v\in\sl(d,\C)$.
\end{lemma}

\begin{proof} Since the angles of the attracting/repelling flags of $g$ are bounded below by $C,$ the angles between the spaces in the decomposition 
$$\sl(d,\C)=V^0(g)\oplus W(g)$$ 
are uniformly bounded below. Thus there exists $k>0$ (depending only on $C$ and $d$) such that $\|u\|\geq \|\pi^0_g(u)\|$
for every $u\in\sl(d,\C)$.
In particular, if  $u=\aff{(g,x)}v-v$ we obtain
\begin{equation}
\label{primera}
\|\aff{(g,x)} v-v\|\geq k\big\|\pi^0_g(\aff{(g,x)}v-v)\big\|.
\end{equation} 
Now observe that, since $\Ad(g)|_{V^0(g)}=\mathrm{id}$, we have
 $$\pi^0_g\big(\aff{(g,x)}v-v\big)=\pi^0_g\big(\Ad(g)v+x-v)= \pi^0_g(x),$$ 
It follows from the  definition of $\sf m$ that $$\Ad(\psi)\big(\pi^0_g\big(\aff{(g,x)}v-v)\big)=\sf m(g,x).$$
for all $v\in\sl(d,\C)$, where $\psi$ is an element conjugating $g$ into standard form.
By Lemma  \ref{control-psi} we can chose $\psi$ so that $\|\psi\|<K,$ where $K$ only depends on $C$ and $d$.
Therefore
$$\|\pi^0_g(\aff{(g,x)}v-v)\|\geq K^{-1}\big\|\Ad(\psi)(\pi^0_g(\aff{(g,x)}v-v)\big\|=K^{-1}\|\sf m(g,x)\|,$$ which combining with \eqref{primera} finishes the proof.
\end{proof}

Let $$\sigma_1(g)\geq\cdots\geq\sigma_d(g)$$ be the singular values of $g$ (with respect to the standard Hermitian metric on $\mathbb C^d$).
The \emph{Cartan projection} 
$$\kappa:\PSL(d,\C)\to \mathfrak a^+\quad\text{is given by}\quad  \kappa(g)=\diag(\log \sigma_i(g)).$$
We observe that when the angles between the attracting/repelling flags of $g$ are uniformly bounded below then one can bound the
ratio of the top eigenvalue and the top singular value.

\begin{lemma}
\label{eigandsing}  
Given $C>0$  and $d\in \mathbb N$,  there exists $b>0$ such that if $g\in\PSL(d,\C)$ is loxodromic
and 
$$\min_{i\in\lb1,d-1\rb}\angle(E_i(g),F_{d-i}(g))\ge C,$$
then 
$$\big| \omega_1(\mu(g))-\log\sigma_1(g)\big|<b$$
where $\omega_1(\mu(g))=\log |\lambda_1(\tilde g)|$.
\end{lemma}

\begin{proof}
Lemma \ref{control-psi} implies that there exists $K>1$, which depends only on $C$ and $d$, so that
$\psi g\psi^{-1}$ is in standard form and $||\psi||<K$. Now notice that $|\omega_1(\mu(g))|=\|\psi g\psi^{-1}\|$
and 
$$\sigma_1(g)=||g||\le \|\psi^{-1}\| \cdot \|\psi g\psi^{-1}\|\cdot \|\psi\| \le \|\psi\|^d \cdot \omega_1(\mu(g)) \le K^d |\mu_1(g)|$$
since $\|\psi^{-1}\|\le\|\psi\|^{d-1}$. So, our result hold with $b=d\log K$.
\end{proof}

\subsection{Totally Anosov representations in $\PSL(d,\mathbb C)$}\label{Anosovpreliminaries1} 
Anosov representations were introduced by Labourie \cite{labourie} and have been a very active topic of research ever since. 
The surveys by Kassel \cite{KasselICM} and
Wienhard \cite{WienhardICM}  give an overview of the topic so we will focus on the facts that we need.

Let $\G$ be a finitely generated group and let $|\,|$ denote the word-length on $\G$ associated to some fixed  finite generating set. 
A representation $\rho:\G\to\PSL(d,\C)$ is \emph{totally Anosov} if there exist positive $C,c$ such that for every $k\in\lb1,d-1\rb$ and $\g\in\G$ one has 
$$\frac{\sigma_{k+1}(g)}{\sigma_{k}(g)}\leq Ce^{-c|\g|}.$$
This characterization of totally Anosov representations was established by Kapovich-Leeb-Porti \cite{KLP} and Bochi-Potrie-Sambarino \cite{BPS}. Totally Anosov representations into $\PSL(d,\C)$  are often called Borel Anosov, since they are Anosov with respect to the Borel subgroup of $\PSL(d,\C)$ 

The space of totally Anosov representations will be denoted by 
$$\frak A_{\sf\Delta}(\G,\PSL(d,\C)).$$ 
It is an open subset of the character variety (see \cite{labourie} and \cite{GW}). Moreover, if $\Gamma$ admits a totally Anosov representation,
then $\G$ is necessarily word-hyperbolic (see \cite{KLP}) and there exists a continuous equivariant map $\xi:\partial\G\to\cal F(\C^d)$ such that if $x,y\in\partial\G$ are distinct, 
then $\xi(x)$ and $\xi(y)$ are transverse flags. Moreover, for every infinite order element $\g\in\G$,  its image $\rho(\g)$ is loxodromic and its attracting/repelling flags are, 
respectively, $\xi(\g_+)$ and $\xi(\g_-)$ (see \cite{labourie} and \cite{GW}). In particular, totally Anosov representations into $\PSL(d,\C)$ are natural
generalizations of convex cocompact representations into $\PSL(d,\C)$.

\subsection{A criterion for proper affine actions}
\label{criterion}

If $\rho:\G\to\PSL(d,\C)$ is totally Anosov, a \emph{cocycle} for $\Ad\rho$ is a map $\co:\G\to\sl(d,\C)$ such that for every $\g,\eta\in\G$ one has 
$$\co(\g\eta)=\co(\g)+\Ad\rho(\g)\co(\eta).$$ 
This equation implies that the map 
$$(\rho,\co):\G\to\PSL(d,\C)\ltimes\sl(d,\C)=\mathrm{Aff}(\sl(d,\C))$$
 is a representation and we consider the associated group of affine transformations $\aff{(\rho,\co)}(\G)$ given by 
 $$\aff{(\rho,\co)}(\g)=\aff{(\rho(\g),\co(\g))}.$$ 
 The most common way to construct a cocycle is to consider a smooth family $\{\rho_t:\Gamma\to\PSL(d,\C)\}$ and
 let 
 $$\co(\gamma)=\frac{d}{dt}\Big|_{t=0}\rho_t(\g)\rho_0(\g)^{-1}\in\sl(d,\C).$$
 It is easy to check that $\co$ is a cocycle for $\Ad\rho$.

The  {\em normalized Margulis spectra} is the set 
$$\M(\rho,\co)=\overline{\Big\{\frac{\sf m(\rho(\g),\co(\g))}{\omega_1(\mu(\rho(g)))} :\g\in\G \text{ has infinite order}\Big\}}$$ 
where $\omega_1(\mu(\rho(g)))= \log|\lambda_1(\widetilde{\rho(g)})|$. 
It is always a compact convex subset of $\h.$

The following is a particular case of a result of  Kassel-Smilga \cite{KS}, see also Ghosh \cite{Ghoshdef}. We include a proof for completeness. 

\begin{proposition}[Properness Criteria]\label{notzero} 
Suppose that $\rho:\Gamma\to\PSL(d,\mathbb C)$ is totally Anosov and $\co:\Gamma\to\sl(d,\C)$ is a cocycle for $\rho$.
If $0\notin\M(\rho,\co),$ then $\aff{(\rho,\co)}(\G)$ acts properly discontinuously on $\sl(d,\C)$.
\end{proposition}

\begin{proof} We follow the outline of Smilga \cite{Smig}. 
We simplify notation by letting $\aff{}:=\aff{(\rho,\co)}$ throughout the proof. 
Since $0\notin\M(\rho,\co),$ there exists $a>0$ so that
$$\|m(\rho(\g),\co(\g))\|\ge a\omega_1(\mu(\rho(\g)))$$
for all $\g\in\G$.

Guichard and Wienhard \cite[Thm. 5.10]{GW} used work of Abels, Margulis and Soifer \cite{AMS} to show that if $\rho$ is 
totally Anosov, then there exists a finite subset $A\subset\G$  and $C>0$ so that if $\gamma\in \G$, then there exists $\alpha\in A$ so that 
$\rho(\alpha\g)$ is loxodromic and
\begin{equation}
\label{alsocrucial}
\min_{i\in\lb1,d-1\rb}\angle(E_i(\alpha\g),F_{d-i}(\alpha\g))\ge C.
\end{equation}
Lemmas \ref{minimo} and \ref{eigandsing} then imply that there exists $b,c>0$ so that for all 
$\g\in\G$ there exists $\alpha\in A$ so that
\begin{equation}
\label{crucial}
\|\aff{}(\alpha\g)v-v\|\geq c\big\|\sf m\big(\rho(\alpha\g),\co(\alpha\g)\big)\big\| \text{ for all } v\in\sl(d,\C) \quad\text{and}\quad
|\omega_1(\mu(\rho(\alpha\g))-\log\sigma_1(\rho(\alpha\g)|<b.
\end{equation}

Let $K\subset \sl(d,\C)$ be a compact set and $S = \{\gamma\ | \ \aff{} (\g)K\cap K\neq\emptyset\}$. 
We consider the compact set 
$$K'=\bigcup_{\alpha\in A}\aff{}(\alpha)(K).$$ 
If $\gamma \in S$, then we can choose $\alpha\in A$ so that $\rho(\alpha\g)$ is loxodromic and satisfies Equation \eqref{crucial}.
Since $\gamma\in S$, we see that 
$\aff{} (\alpha\g)(K)\cap \aff{}(\alpha)(K)\neq\emptyset$, so $\aff{}(\alpha\g)K\cap K'\neq\emptyset.$ 
Therefore, combining Equations \eqref{alsocrucial} and \eqref{crucial}, we see that
$$\omega_1(\mu(\rho( \alpha\g))) \leq \frac{1}{ac}\inf_{v\in\sl(d,\C)} \|\aff{}(\alpha\g)v-v\| \leq \frac{1}{ac}\mbox{diam}(K\cup K') = C_1,$$ 
so, again by Equation \eqref{crucial},
$$\log\sigma_1(\rho(\alpha\g)) <b+C_1.$$
If we set $d=\max_{\alpha\in A}\log\sigma_1(\alpha)$, then we conclude that
$$\log\sigma_1(\rho(\g)) <b+C_1+d.$$
Since there are only finitely many elements of $\Gamma$ such that  $\log\sigma_1(\rho(\g)) <b+C_1+d$, we
conclude that  $S$ is finite and so $F$ is proper.

\end{proof}

\medskip\noindent
{\bf Remark:} If $d=2$ and $\rho$ is quasifuchsian  it is easy to use a ping-pong argument to produce two elements $\alpha_1$ and $\alpha_2$ 
so that if $\gamma\in \rho(\pi_1(S))$, then the attracting and repelling fixed points of $\alpha_i\gamma$ are uniformly separated for either $i=1$ or $i=2$.

\subsection{Proper affine actions on $\sl(2,\C)$}
\label{conclusion}
In order to apply our Properness criterion and prove Theorem \ref{Main Theorem 3},
we must be able to compute the Margulis invariant. We first introduce some notation.

If $\tilde g\in \SL(d,\C)$ is loxodromic,  we define 
$$\lambda:\SL(d,\C)\to\SL(d,\C)\quad \text{by}
\quad\lambda(\tilde g)=\mathrm{diag}\big(\lambda_1(\tilde g),\ldots,\lambda_d(\tilde g)\big).$$
Recall that $\lambda$ is an analytic function on the open set of loxodromic elements in $\SL(d,\C)$.
If $g\in\PSL(d,\C)$ is loxodromic and $\dot g\in\sf T_g\PSL(d,\C)$, let $\{g_t\}_{t\in (-\epsilon,\epsilon)}$ be a smooth family of 
loxodromic elements  in $\PSL(d,\C)$ so that $g_0=g$ and $\frac{d}{dt}\big|_{t=0} g_t=\dot g$. One can lift
$\{g_t\}_{t\in (-\epsilon,\epsilon)}$ to a smooth family $\{\tilde g_t\}_{t\in (-\epsilon,\epsilon)}$  of loxodromic elements
of $\SL(d,\C)$. We abuse notation by defining
$$\big({\rm {d}}\lambda(\dot g)\big)\lambda(g)^{-1}=\left({\rm {d}}\lambda\left(\frac{d}{dt}\Big|_{t=0} \tilde g_t\right)\right)\lambda\big(\tilde g_0^{-1}\big)\in\h.$$
Since any two lifts differ by multiplication by an element of $Z_d$, this is well-defined.
We also let
$$x_g:=\dot g g^{-1}\in\sl(d,\C).$$

The following computation was first performed by Goldman and  Margulis \cite{goldman-margulis} when $\sf G=\mathsf{SO}(2,1)$.
We give a version which appears in  Sambarino \cite[Cor. 8.3]{sambarino}. 
Similar computations can also be found in Ghosh \cite{Ghoshdef,avatars}, Kassel-Smilga \cite{KS} and
Danciger-Zhang \cite{DZ}. We include a proof for completeness.

\begin{proposition}
\label{eigenvaluevariation} 
If $g\in\PSL(d,\C)$ is  loxodromic and $\dot g\in\sf T_g\PSL(d,\C),$ then 
$$\sf m(g,x_g)=\big({\rm {d}}\lambda(\dot g)\big)\lambda(g)^{-1}$$
In particular,
 $$\Re\Big(\sf m(g,x_g)\Big)=\dd\mu(\dot g).$$
\end{proposition}

\begin{proof}
Let $\{g_t\}_{t\in (-\epsilon,\epsilon)}$ be a smooth family of loxodromic elements  in $\PSL(d,\C)$
so that $g_0=g$ and $\frac{d}{dt}\big|_{t=0} g_t=\dot g$. Lift
$\{g_t\}_{t\in (-\epsilon,\epsilon)}$ to a smooth family of loxodromic elements
of $\SL(d,\C)$  which we will further abuse notation by continuing to call $\{g_t\}_{t\in (-\epsilon,\epsilon)}$.
We may choose a smooth family $\{\psi_t\}_{t\in (-\epsilon,\epsilon)}$ in $\SL(d,\mathbb C)$
so that  $\mu(\g_t) = \psi_t  g_t\psi_t^{-1}$ for all $t$.
Let  $\psi=\psi_0$ and $\psi_t = \psi\phi_t$, so $\phi_0 = \Id$ and $\frac{d}{dt}\big| \phi_t=\dot\phi_0\in\frak{sl}(d,\C)$.
Differentiating the expression $\lambda(g_t) = \psi (\phi_t  g_t\phi_t^{-1})\psi^{-1}$ gives
$$\dd\lambda(\dot{g})  = \psi(\phi_0\dot g\phi_0^{-1}+\dot\phi_0 g_0- g_0\dot\phi_0)\psi^{-1} = 
\psi(\dot g+\dot\phi_0 g_0- g_0\dot\phi_0)\psi^{-1},$$
so
$$\dd\lambda(\dot{g})\lambda(g)^{-1} = \psi(\dot g +\dot\phi_0 g_0-  g_0\dot\phi_0) g_0^{-1}\psi^{-1} = 
\psi(\dot g g_0^{-1} + \dot\phi_0- g_0\dot\phi_0 g_0^{-1})\psi^{-1}.$$
Therefore
$$\dd\lambda(\dot{g})\lambda(g)^{-1} =  \Ad(\psi)(x_g + \dot\phi_0-\Ad(g)(\dot\phi_0)).$$

Recall that the  map $\pi^0_g$  is projection onto the 1-eigenspace $V^0(g)$ of $\Ad(g)$ parallel to  the sum $W(g)$ of the other eigenspaces.
Thus
$$\pi_{\psi g\psi^{-1}}^0  = \Ad(\psi)\circ \pi^0_g\circ \Ad(\psi^{-1}).$$ 
Since  $\psi g\psi^{-1}$ is diagonal and loxodromic, $\pi^0_{\psi g\psi^{-1}} = \pi^0$ where $\pi^0:\sl(d,\C)\rightarrow \frak h$ 
is  the projection to the diagonal obtained by changing all the non-diagonal entries to $0$. Therefore  
$$\pi^0 \circ \Ad(\psi) = \Ad(\psi)\circ \pi^0_g.$$
Applying $\pi^0$ to  both sides gives
$$\dd\lambda(\dot{g})\lambda(g)^{-1}= \pi^0(d\lambda(\dot{g})\lambda(g)^{-1}) = \Ad(\psi)\Big(\pi_g^0\big(x_g+\dot\phi_0-\Ad(g)(\dot\phi_0)\big)\Big).$$
Since $\pi_g^0\circ \Ad(g) = \pi_g^0$, we conclude that
 $$\mathrm d\lambda(\dot{g})\lambda(g)^{-1} = \Ad(\psi)(\pi_g^0(x_g)) = \sf m(g,x_g).$$

Our formula for $\Re\Big(\sf m(g,x_g)\Big)$ follows from the elementary fact that if $(z_t)_{t\in(-\eps,\eps)}$ is a smooth curve in $\C^*$,
then $\Re\big(\dot z/z\big)=\frac{\partial}{\partial t}\big|_{t=0}\log|z_t|.$
\end{proof}

With this computation in hand, we can combine Theorem \ref{Main Theorem 4} and Theorem \ref{properness criterion} 
to establish Theorem \ref{Main Theorem 3}.

\medskip\noindent
{\bf Theorem \ref{Main Theorem 3}.} {\em
If $[\rho]\in \QF(S)$ is not fuchsian and both $\Omega_+(\rho)$ and $\Omega_-(\rho)$ are moderately bent,
then there exists a representation $\sigma:\pi_1(S)\to\mathrm{Aff}(\sl(2,\mathbb C))$  whose linear part is $\Ad(\rho)$
so that $\sigma(\pi_1(S))$ acts properly discontinuously on $\sl(2,\C)$.}

\medskip

\begin{proof}
Let $w=w(\rho)=w_+(\rho)+w_-(\rho)$.
Let $\{\rho_t\}_{t\in (-\epsilon,\epsilon)}$ be a smooth family in $\mathrm{Hom}\big(\pi_1(S),\PSL)2,\C)\big)$ such that $\rho_0=\rho$ and 
$$\frac{d}{dt}\Big|_{t=0} \rho_t=w$$ 
Let $\co$ be the cocycle for $\Ad\rho$ given by 
 $\co(g)=\frac{d}{dt}\Big|_{t=0}\rho_t(g)\rho_0(g)^{-1}.$
Theorem \ref{Main Theorem 4} implies that there exists a $K > 0$ so that
$$\dd\ell_\g(w) \leq  -K \ell_\g(\rho)$$
for all $\g\in\pi_1(S)$. Since $\ell_\rho(\gamma)=2\log(\mu_1(\gamma))$, Proposition \ref{eigenvaluevariation}
implies that 
$$\Re\Big(\sf m(\rho(g),\co(g))\Big)=\Big(\dd\ell_\g(\co) ,-\dd\ell_\g(\co)\Big) $$
and so the first coordinate of $\frac{\Re\big(\sf m(\rho(g)\co(g))\big)}{\omega_1(\rho(g))}$ is always less than $-K$. 
Therefore, $0\notin\M(\rho,\co)$, so our Properness Criterion, 
Proposition \ref{notzero}, implies that the action $F_{\rho,\co}$ is proper. 
\end{proof}

\subsection{Proper affine actions on other complex Lie groups}
\label{others}
We first observe that one may generalize the proof of Theorem \ref{Main Theorem 3} into the setting of totally Anosov
representations into $\PSL(d,\C)$.
Let $\a$ denote the real part of $\h$ (i.e. the matrices in $\h$ with real entries) and let $\a^*$ be its dual space.
Suppose that $\Gamma$ is a finitely generated group.
If $\varphi\in\a^*$ and $\gamma\in\Gamma$,  consider the function 
$$\varphi^{\g}:\mathrm{Hom}\big(\G,\PSL(d,\C)\big)\to\a\quad\text{given by}\quad \varphi^{\g}:\rho\mapsto\varphi\big(\mu\big(\rho(\g)\big)\big).$$ 
If $(\rho_t)_{t\in(-\eps,\eps)}:\G\to\PSL(d,\C)$ is a smooth curve in $\frak X\big(\G,\PSL(d,\C)\big)$ so that $\rho_0$ is totally Anosov,
then $\varphi^\gamma$ is differentiable at 0, whenever $\gamma\in\Gamma$ has infinite order.

\begin{corollary}\label{properness criterion} Let $(\rho_t)_{t\in(-\eps,\eps)}$ be a smooth curve in $\mathrm{Hom}\big(\G,\PSL(d,\C)\big)$
such that $\rho=\rho_0$ is totally Anosov and let 
$\co:\G\to\frak g$ be the  cocycle given by  $\co(g)=\frac{d}{dt}\Big|_{t=0}\rho_t(g)\rho_0(g)^{-1}.$
If there exists $\varphi\in\a^*$ such that
$$\dd\varphi^\g (\co)\geq\omega_1(\mu(\rho(\g)))$$ 
for all $\gamma\in\Gamma$, then $\aff{(\rho,\co)}(\Gamma)$ acts properly discontinuously on $\sl(d,\C)$ is proper.
\end{corollary}

\begin{proof}
By Proposition \ref{notzero} it suffices to show that $0\notin\M(\rho,\co).$ 
For $\g\in\G$ we have, by Proposition \ref{eigenvaluevariation} and our assumption imply that
$$\varphi\big(\Re\big(\sf m(\rho(\g),\co(\g))\big)\big)=\dd\varphi^{\g}(\co)\geq \omega_1(\mu(\rho(\g))).$$ 
Thus $\varphi\Big(\Re\big(\M(\rho,\co)\big)\Big)\subset[1,\infty),$ which implies that $0\notin\M(\rho,\co).$ 
\end{proof}

We now consider other complex simple Lie groups. If $\sf G$ is a complex simple Lie group with Lie algebra $\frak g$,
the concepts from the previous sections,  such as the Cartan algebra $\frak h$,
the Jordan projection $\mu:\sf G\to\frak a=\Re(\h),$  and the Cartan projection $\kappa:\sf G\to\frak a=\Re(\h)$ extend to this setting.  
Finally the Margulis invariant $\sf m(g,x)\in\h$ of $(g,x)\in\sf G\ltimes\frak g,$ can be extended to this more general setting. 
It will suffice for our purposes to recall that (after identifying $\sf G\ltimes\frak g$ with $\sf T \sf G$)   for $(g,x) = \dot g \in \sf T_g \sf G$,   
$$\Re\big( \sf m(g,x)\big) = \dd\mu(\dot g)$$
(see \cite[Corollary 8.3]{sambarino}).  
We refer the reader to the book of Benoist and Quint \cite{BQ-book} for the standard Lie algebra and Lie group concepts 
and to Smilga \cite{Smig} for a discussion of the Margulis invariant in this more general context.

In this setting it is convenient to consider the \emph{first fundamental restricted weight} $\varpi_1\in\frak a^*$ of $\sf G.$ 
The number $\varpi_1(\mu(g))$ replaces $\omega_1(\mu(g))$ (and coincides with it when $\sf G=\PSL(d,\C)$). 
We say that a representation $\rho:\G\rightarrow \sf G$ is totally Anosov if it is Anosov with respect to a minimal parabolic
subgroup of $\sf G$ (see \cite{GW}).
Let $ \frak A_{\sf\Delta}(\G,\sf G)$  be the subspace of the character variety $\frak X(\G,\sf G)$ consisting of totally Anosov representations. 

For a homomorphism $(\rho,u):\Gamma \to\sf G\ltimes\frak g$ the normalized Margulis spectrum is the subset of $\h$,
$$\M(\rho,\co)=\overline{\Big\{\frac{\sf m(\rho(\g),\co(\g))}{\varpi_1\big(\mu(\rho(\g))\big)}:\g\in\G\text{ has infinite order }\Big\}}.$$
We recall that the Margulis spectrum varies continuously.

\begin{proposition}[{Sambarino \cite[Lemma 2.34]{sambarino}}]\label{continuous}
The map $(\rho,\co)\mapsto \M(\rho,\co)$ is continuous on the open subset of $\mathcal C(\G,\sf G)$ where 
$\rho \in \frak A_{\sf\Delta}(\G,\sf G)$ (where the image lies in the space of compact subsets of $\h$ with the Hausdorff topology).
\end{proposition}

We are now ready to prove part (1) of Corollary \ref{complexgroups}. Recall that $\mathcal C(S,\sf G)$ is the space of pairs $(\rho,\co)$ where
$\rho:\pi_1(S)\to\sf G$ is a representation and $\co $ is a cocycle for $\Ad\rho$.

\begin{theorem}\label{complexgroups1} Let $\sf G$ be a complex simple Lie group with Lie algebra $\frak g$, then the
subset of $\mathcal C(S,\sf G)$  consisting of pairs $(\rho,\co)$
so that $F_{\rho,\co}(\pi_1(S))$ act properly on  $\frak g$ has non-empty interior.
\end{theorem}

\begin{proof}
We first give the proof in the case that $\sf G=\mathrm{Inn}(\frak g)$ is the group of inner autmorphisms of $\frak g$.
Theorem \ref{Main Theorem 3} implies there exists $\rho\in \QF(S)$ and a cocycle $\co$ for $\Ad\rho$
such that $0\notin\Re(\M(\rho,\co)).$  
Kostant defined an embedding $\tau:\PSL(2,\C)\to\mathrm{Inn}(\frak g)$,  called the principal embedding, which is unique up to conjugation.
If $\sf G = \PSL(d,\C)$, the principle embedding is the irreducible embedding (see Kostant \cite{Kostant}).  
The principal embedding $\tau:\PSL(2,\C)\to\mathrm{Inn}(\frak g)$  has the property that  
 $$\varpi_1(\mu(\tau(g))) = c_{\frak g} \omega_1(\mu(g))$$
 for all $g\in\PSL(2,\C)$,
where $c_{\frak g}>0$ is an explicit constant ($c_\frak g=d-1$ when $\frak g=\sl(d,\C)$).
The principal embedding $\tau$ induces an embedding 
$$\mathrm d\tau:\sl(2,\C)\to\frak g\quad\text{so that}\quad
\varpi_1\big(\Re(\M(\tau(\rho),\dd\tau\circ\co)\big)=\omega_1( \Re(\M(\rho,\co))),$$
so $0\notin\Re(\M(\tau(\rho),\dd\tau(\co))).$

Since the map $(\eta,{\vec v})\mapsto \M(\eta,{\vec v})$ is continuous when $\eta$ is totally Anosov 
and $0\notin\varpi_1(\Re(\M(\tau(\rho),\dd\tau(\co))))$, we conclude that the same holds for nearby $(\eta,{\vec v}).$ 
In particular, $0\notin\M(\eta,{\vec v})$ and thus the analogue of  Proposition \ref{properness criterion}  (see \cite{KS} or \cite[Prop. 7.2]{sambarino})
shows that $\aff{(\eta,\vec v)}(\pi_1(S))$ acts properly discontinuously on $\frak g$.

If $\sf G$ is not $\mathrm{Inn}(\frak g)$, then there is a covering map $p:\sf G\to \mathrm{Inn}(\frak g)$ and $\tau$ lifts to a
homomorphism $\tilde \tau:\SL(2,\C)\to \sf G$ such that  $\varpi_1(\mu(\tilde\tau(g))) = c_{\frak g} \omega_1(\mu(g))$.
Recall that every quasifuchsian representation
$\rho:\pi_1(S)\to\PSL(2,\C)$ lifts to a representation \hbox{$\tilde \rho:\pi_1(S)\to\SL(2,\C)$} (see Culler \cite{culler}).
Since $\Ad\tilde\rho=\Ad\rho$ and $\mathrm d\tilde\tau \circ u=\mathrm d\tau\circ u$,
we may complete the proof in this case using the same argument.
 \end{proof}

\section{Proper actions on the group manifold}

In this section, we study actions on the group manifold  of a Lie group via left and right multiplication. The final goal of the
section is to establish part (2) of Corollary \ref{complexgroups}.

The action of $\sf G\times\sf G$  on $\sf G$ via $(g,h)x=gxh^{-1}$ is transitive and the stabilizer of $e\in\sf G$ is the diagonal embedding 
$${\rm{Diag}}(\sf G)=\{(g,g):g\in\sf G\}\subset\sf G\times\sf G.$$ 
The {\em group manifold} of $\sf G$ is the quotient  $\sf G\times\sf G/\rm{Diag}(\sf G).$
The quotient may be identified with $\sf G$. 
The group manifold inherits a pseudo-Riemannian metric from the Killing form on $\mathfrak g$ which is invariant under right
and left multiplication by $\sf G$.
The most classical situation occurs when $\sf G=\PSL(2,\R)$, in which case
the group manifold is three-dimensional anti-de Sitter space.

Let $\Gamma$ be a finitely generated group. Given two representations $\rho,\eta:\G\to\sf G$, we let $\rho\times\eta:\G\to\sf G\times\sf G$ 
be the product representation  defined by $\g\mapsto(\rho(\g),\eta(\g)).$ The quotient action  of $\Gamma$ on the group manifold
is simply the action  given by right multiplication by $\rho(\g)$ and left multiplication by $\eta(\g)^{-1}$.

\subsection{Properness criteria}
Benoist\cite{BenoistK-criterion} and Kobayashi \cite{BKobayashi-criterion} developed a general criterion to determine when an action
on reductive homogeneous spaces are proper. When restricted to our situation we get the following criterion.

\begin{theorem}[{Benoist-Kobayashi properness criterion}]
\label{BK}
Let $\rho,\eta:\G\to\sf G$ be two representations and assume that $\rho\times\eta$ has discrete image and finite kernel. 
If for every $K>0$, the set 
$$\big\{\g\in\G:\|\kappa(\rho(\g))-\kappa(\eta(\g))\|\leq K\big\}$$ 
is finite, then $\rho\times\eta(\G)$ acts properly on the group manifold $\sf G\times\sf G/\rm{Diag}(\sf G)$.
\end{theorem}

Their properness criterion is most commonly used in the following form.

\begin{corollary}
\label{crit} 
Let $\rho,\eta:\G\to\sf G$ be two representations of a finitely generated group and assume that $\rho\times\eta$
has discrete image and finite kernel.  If there exists $b>1$ and $c\geq0$ such that 
$$\varpi_1(\kappa(\rho\g))-b\varpi_1(\kappa(\eta\g))\geq -c,$$ 
for every $\g\in\G$, then $\rho\times\eta(\G)$ acts properly on the group manifold $\sf G\times\sf G/\rm{Diag}(\sf G)$.
\end{corollary}

\begin{proof} We consider the functional $\varphi_\mu:\a\times\a\to\R$ defined by $\varphi_\mu(x,y)=\varpi_1(x)-b\varpi_1(y)$ and an arbitrary $K>0.$ 
Since $\mu>1$ and $c\geq0,$ the set $\{\varphi_\mu\geq -c\}\cap \rm{Diag}(\a^+)$ is compact and thus also is its intersection with the tubular neighborhood 
$$\big\{(x,y)\in\a^+\times\a^+:\|x-y\|\leq K\big\}.$$ 
Since $\rho\times\eta$ is a proper map (by assumption) and $\cartan$ is also proper, 
we conclude that the set $\{\g\in\G:\|\kappa(\rho(\g))-\kappa(\eta(\g))\|\leq K\}$ is finite and thus Benoist-Kobayashi's criterion applies and we may
conclude that  $\rho\times\eta(\G)$ acts properly.
\end{proof}

\subsection{Ledrappier potentials for Anosov representations}
\label{Anosov2}
We now develop some more structure theory for Anosov representations, see Bridgeman-Canary-Labourie-Sambarino \cite{pressure}.
Given a totally Anosov representation $\rho:\G\to\sf G$ there exists a compact metric space $\sf U_\rho\G$ equipped with a flow 
$\phi^\rho=(\phi^\rho_t:\sf U_\rho\G\to\sf U_\rho\G)_{t\in\R}$ whose periodic orbits, if $\G$ is torsion-free,  are in one-to-one correspondence 
with conjugacy classes of  infinite order elements of $\G$ such that the period of $[\g]$ is $\varpi_1(\mu(\rho\g))$. If $\G$ has torsion 
the correspondence is finite-to-one (see Blayac \cite{blayac} and Carvajales \cite[Appendix]{carvajales} for more information).

If $\eta:\G\to\sf G$ is also totally Anosov, there exists a H\"older-continuous function $f_{\rho\eta}:\sf U_\rho\G\to\R_+$ 
such that for every conjugacy class $[\g]$ the integral of $f_{\rho\eta}$ over the associated periodic orbit is 
$$\int_{[\g]}f_{\rho\eta}=\varpi_1(\mu(\eta\g)).$$
The function $f_{\rho\eta}$ is called \emph{the Ledrappier potential} of $\eta$ (seen from $\rho$) and the map $\eta\mapsto f_{\rho\eta},$ 
well defined up to Liv\v sic-cohomology, is called the \emph{thermodynamic mapping}. For $\alpha\in(0,1)$, let $
\textrm{H\"ol}^\alpha(\sf U_\rho\G,\R)$ be the space of H\"older-continuous maps with exponent less than $\alpha,$ equipped with the standard 
norm that makes it a Banach space.

If $\eta\in\frak A_{\sf\Delta}(\G,\sf G)$ is a smooth point of $\mathrm{Hom}(\pi_1(S),\sf G)$, then there exists a neighborhood $\cal U_\eta$ of $\eta$ 
and $\alpha\in(0,1)$  such that the thermodynamic mapping sends $\cal U_\eta$ into $\textrm{H\"ol}^\alpha(\sf U_\rho\G,\R)$ 
and it is an analytic map when restricted to $\cal U_\eta$ (see \cite[Proposition 6.2]{pressure}).

\subsection{Proper affine actions on the group manifold}

We are now ready to establish a 
weaker version of a recent result due to Ghosh-Kobayashi \cite{GK},  which suffices for our purposes. We include a proof for completeness.

If $U$ is a smooth open set in $\mathrm{Hom}(\Gamma,\sf G)$, then one can identify $\sf TU$ with an
open set in $\mathcal C(\Gamma,\sf G)$ and there exists a smooth map $\Psi:\sf TU\to \mathrm{Hom}(\Gamma,\sf G)$
with the property that $\frac{d}{dt}\Big|_{t=0}\Psi(\rho,tv)=v$ for all $v\in \sf T_\rho\mathrm{Hom}(\Gamma,\sf G)$ with $\rho\in U$.
We will use the notation $(\rho,v)$ to denote a tangent vector 	$v\in \sf T_\rho\mathrm{Hom}(\Gamma,\sf G)$.

\begin{proposition}\label{GxG} 
Suppose $\rho_0\in\frak A_{\sf \Delta}(\G,\sf G)$  is a smooth point of $\mathrm{Hom}(\Gamma,\sf G)$,  $v_0\in \sf T_{\rho_0}\mathrm{Hom}(\Gamma,\sf G)$
and $\co_{v_0}\in\mathcal C(\Gamma,\sf G)$ is the cocycle for $\Ad\rho$ associated to $v_0$.
If $\rho_0$ is semi-simple and  $0\notin\varpi_1\big(\Re(\M(\rho_0,\co_{v_0}))\big)$, then there exists an open neighborhood $V$ of $(\rho_0,v_0)$ in $\sf TU$ and $\delta>0$ 
so that if $(\rho, v)\in V$ and $t\in (0,\delta)$, 
the group $\rho\times \Psi(\rho,tv)$ acts properly on $\sf G\times\sf G/\rm{Diag}(\sf G).$ 
\end{proposition}

\begin{proof} 
Since $\frak A_{\sf \Delta}(\G,\sf G)$ is open, we may assume that $\rho$ is totally Anosov for all $(\rho,v)\in V$.
We may also assume that the set 
$$V_0=\{(\rho,v)\in V\text{ for some }v\in \sf T_\rho\mathrm{Hom}(\Gamma,\sf G)\}$$
is a smooth open set in $\mathrm{Hom}(\Gamma,\sf G)$.
We then consider the Ledrappier potential $f_t(\rho,v):=f_{\rho \Psi(\rho,tv)}:\sf U_\rho\G\to\R$ from \S\ref{Anosov2},
with respect to the base representation $\rho.$ 
Since the map $(\rho,\eta)\mapsto f_{\rho\eta}$ is analytic  on an open neighborhood of $(\rho,\rho)$
(which may assume contains $V_0\times \Psi(V)$),
we may find $K>0$ and $\delta>0$ such that, after possibly further shrinking $V$
one has 
$$\|f_t(\rho,v)-f_0-t\frac{d}{dt}\Big|_{t=0}f_{\rho\Psi(\rho,tv)}\|_\infty\leq K t^2$$
for every $(\rho,v)\in V$ and $t\in [0,\delta]$. 
By integrating the Ledrappier potential over periodic orbits
of the geodesic flow,
one observes  that
$$\big|\varpi_1(\mu(\Psi(\rho,tv)(\g)))-\varpi_1(\mu(\rho(\g)))-t\dd\varpi_1^\g(v)\big|\leq Kt^2\varpi_1(\mu(\rho(\g)))$$  
for every $(\rho,v)\in V$,  $t\in [0,\delta]$ and infinite order element $\g\in\G$. 
Since $\varpi_1\big(\Re(\M(\rho,\co_v))$ varies continuously, we may assume
that there exists $c>0$ so that $\varpi_1\big(\Re(\M(\rho,\co_v))\big)\leq-c<0$ for all $(\rho,v)\in V$, so 
$$(1-ct+K t^2)\varpi_1(\mu(\rho(\g)))\geq\varpi_1(\mu(\Psi(\rho,tv)(\g))).$$
for every $(\rho,v)\in V$,  $t\in [0,\delta]$ and infinite order element $\g\in\G$. 
If we assume that $\delta\le\frac{c}{2K}$, then
$$\varpi_1(\mu(\rho(\g)))\geq b_t\varpi_1(\mu(\Psi(\rho,tv)(\g)))$$ 
for all $t\in (0,\delta]$ and $\g\in\G$
with $b_t=\frac{1}{1-ct+K t^2}> 1.$

Benoist \cite[\S 4.5]{benoist-asymptotic} showed there exists a finite set $A\subset\Gamma$ so that for all $\gamma\in\Gamma$,
there exists $\alpha\in A$ so that $\rho_0(\alpha\gamma)$ is $(\Delta,2\epsilon)$-proximal. 
(Recall that $g\subset \sf G$ is $(\Delta,\delta)$-proximal if it is loxodromic, if $g^-$ and $g^-$ are its attracting and repelling flags, then $d(g^+,g^-)>\delta$, 
$g(B_\delta(g^-))\subset b_\delta(g^+)$ and $g$ is $\delta$-Lipschitz on $b_\delta(g^+)$ where $B_\delta(g^-)$ is the complement of the open
neighborhood of $g^-$ of radius $\delta$ and $b_\delta(g^+)$ is the closed neighborhood of $g^+$ of radius $\delta$.)
After possibly shrinking $V$ and $\delta$, one may assume that if $\g\in\G$, then
there exists $\alpha\in A$ so that if $(\rho,v)\in V$ and $t\in [0,\delta)$, then
$\Psi(\rho,tv)(\alpha\gamma)$ is $(\Delta,\epsilon)$-proximal. Benoist \cite[\S 4.5]{benoist-asymptotic}
also showed that  there exists a compact subset $N$ of $\frak a$, depending only on $\epsilon$, so that
$g\in \sf G$ is $(\Delta,\epsilon)$-proximal, then 
$\mu(g)-\kappa(g)\in N$.
Therefore, there exists $E>0$ so that if $g\in \sf G$ is $(\Delta,\epsilon)$-proximal, then 
$$\Big|\varpi_1\big(\mu(g)\big)-\varpi_1\big(\kappa(g)\big)\Big|<E.$$
By \cite[Cor. 6.34]{BQ-book}, if $g,h\in\sf G$, then
$$\varpi_1\big(\kappa(gh))\big)\le\varpi_1\big(\kappa(g)\big)+\varpi_1(\kappa(h)).$$
Let $S=\max\{\omega_1(\rho_t(\alpha))\ :\ \alpha\in A\mbox{ or } \alpha^{-1}\in A,\ t\in [0,\delta]\}$. 
If $(\rho,V)\in U$, $t\in[0,\delta]$
and $\gamma\in \Gamma$, then there exists $\alpha\in A$ so that $\rho(\alpha\gamma)$ and $\Psi(\rho,tv)(\alpha\gamma)$ are both
$(\Delta,\epsilon)$-proximal, so
\begin{eqnarray*}
\varpi_1(\kappa(\rho(\g))) & \ge & \varpi_1(\kappa(\rho(\alpha\g)))-S\\ 
&\ge &\varpi_1(\mu(\rho(\alpha\g)))-(S+E)\\
&\geq & b_t\varpi_1(\mu(\Psi(\rho,tv)(\alpha\g)))-(S+E)\\
&\geq & b_t\Big(\varpi_1(\kappa(\Psi(\rho,tv)(\g)))-(S+E)\Big)-(S+E)\\
&\geq &b_t\varpi_1(\kappa(\Psi(\rho,tv)(\g)))-\mu_t(S+E)-(S+E).
\end{eqnarray*}
Since $b_t>1$ if $t\in (0,\delta)$, 
Corollary \ref{crit} implies that $\rho\times\Psi(\rho,tv)\,(\G)$ acts properly on $\sf G\times\sf G/\rm{Diag}(\sf G)$  for all
$(\rho,v)\in V$ and  $t\in(0,\delta).$
\end{proof}

Let ${\P}(S,\sf G)\subset\mathrm{Hom}(\pi_1(S),\sf G)$ be the set of totally Anosov representations $\rho:\pi_1(S)\to\mathsf G$ so
that there exists a cocycle $\co$ for $\Ad\rho$ so that $0$ does not lie in $\varpi_1\big(\Re(\M(\eta,\co))\big).$
Let  ${\P}_{sm}(S,\sf G)$ denote the set of semisimple representations $\rho\in{\P}(S,\sf G)$ which are smooth points of $\mathrm{Hom}(\pi_1(S),\sf G)$.

We first observe that ${\P}_{sm}(S,\sf G)$ is non-empty.

\begin{lemma} 
\label{smoothness}
Let $\sf G$ be a complex simple Lie group and $\tilde\tau:\SL(2,\C)\to \sf G$ be the lift of a principal embedding $\tau:\PSL(2,\C)\to \mathrm{Inn}(\frak g)$.
If $\rho:\pi_1(S)\to \SL(2,\C)$ is a lift of a quasifuchsian representation which is not fuchsian, then $\tilde\tau\circ \rho$ is a smooth point of $\Hom(\pi_1(S),\sf G).$
In particular, ${\P}_{sm}(S,\sf G)$ is non-empty.
\end{lemma}

\begin{proof}
Goldman \cite[\S 1.2 \& 1.3]{goldman-symplectic} showed that  $\tilde\tau\circ \rho$ is a smooth point if the Lie algebra  of its centralizer  is zero-dimensional. 
Since lying in the centralizer is an algebraic condition, it suffices to show the Lie algebra of its Zariski closure  is zero-dimensional. 
If $\rho:\pi_1(S)\to\PSL(2,\C)$ is quasifuchsian, but not fuchsian, the Lie algebra of the  Zariski closure of $\tilde\tau\rho(\pi_1(S))$ is $\frak s = d\tau(\sl(2,\C))$.

Suppose that $x$ lies in the centralizer of $\frak s$. By Kostant (see  \cite[Lemma 5.2]{Kostant}),  $\frak s$ is an $\frak\sl(2,\C)$ triple $\{h, e, f\}$  where $h$ is in the interior of a Weyl chamber  with Cartan subalgebra $\frak h$ and  simple roots $\Delta$ such that 
$$e = \sum_{\alpha \in \Delta} e_\alpha \qquad e_\alpha \in \mathfrak g_\alpha, \ e_\alpha \neq 0.$$ 
Thus the centralizer of $h$ is $\frak h$ and therefore $x \in \frak h$.   As $x$ is also in the centralizer of $e$,
$$0 = \ad(x)(e) = \sum_{\alpha \in \Delta} \ad(x)(e_\alpha) = \sum \alpha(x) e_\alpha.$$
It follows that $\alpha(x) = 0$ for all $\alpha \in \Delta$ and as $\Delta$ is a basis for $\frak h^*$ then $x =0$.
Therefore, $\tilde\tau\circ \rho$ is a smooth point of $\Hom(\pi_1(S),\sf G).$

Theorem \ref{Main Theorem 4} implies that there exists a quasifuchsian representation $\sigma:\pi_1(S)\to\mathsf{PSL}(2,\mathbb C)$ which admit a cocycle $\co$ 
so that $0\notin\varpi_1\big(\Re(\M(\rho,\co))\big)$. Let $\rho:\pi_1(S)\to\SL(2,\C)$ be a lift of $\sigma$, then the proof of Corollary \ref{complexgroups1} implies
that $\tilde\tau\circ\rho\in {\P}(S,\sf G)$. The representation $\tilde\tau\circ\rho$ is semi-simple, since its Zariski closure is simple and we have
just shown it is a smooth point of $\Hom(\pi_1(S),\sf G)$, so $\tilde\tau\circ\rho\in {\P}_{sm}(S,\sf G)$.
\end{proof}

We now establish the following more precise version of part ii) of Corollary \ref{complexgroups}.

\begin{corollary} Suppose that $\sf G$ is a  complex simple Lie group. 
The space of representations of $\pi_1(S)$ into $\sf G\times\sf G$ such that the action on $\sf G\times\sf G/\rm{Diag}(\sf G)$ 
is proper has non-empty interior and its closure contains the set $\{\rho\times\rho:\rho\in\P_{sm}(S,\sf G)\}.$
\end{corollary}

\begin{proof} 
For each $\rho_0\in\P_{sm}(S,\sf G)$, there exists a cocycle $\co_0$ for $\Ad\rho_0$ so that 
$0\notin\varpi_1\big(\Re(\M(\rho_0,\co_0)\big)$.
Proposition \ref{GxG} produces $\delta_{\rho_0}>0$ and a neighborhood $V_{\rho_0}$ of $(\rho_0,\co_0)$ in $\mathcal C(\Gamma,\sf G)$
so that if $(\sigma,v)\in V_{\rho_0}$ and $t\in (0,\delta_{\rho_0}]$, then
$\sigma\times\Psi(\sigma,tv)$ yields a proper action of $\pi_1(S)$ on $\sf G\times\sf G/\rm{Diag}(\sf G)$.
The set 
$$\{\sigma\times\Psi(\sigma,tv):  (\sigma,v)\in V_\rho\text{ and }t\in(0,\delta_\rho) \ \text{for some}\ \ \rho\in \P_{sm}(S,\sf G)\}$$
 is a non-empty open set that verifies the desired conditions.
\end{proof}

\section{Angle  and bending bounds}

In order to show that a Jordan domain is moderately bent, it suffices to have a bound on the angle between
geodesics in the intrinsic metric on the boundary of the convex hull and geodesics in $\mathbb H^3$ joining points on the boundary.
Our final goal is to prove Theorem \ref{Main Theorem 5} which shows that if $||\beta_\mu||_L<r(L)$, then $\Omega_\mu$ is moderately bent.

Given  a measured lamination $\mu$ on $\Hp$ one may  construct a locally convex map $f_\mu:\Hp \rightarrow \Hs$ with bending lamination $\mu$.
The construction may be formulated in the same language as the construction of a bending deformation.  Specifically, for a locally finite 
lamination $\mu$, we consider the cocycle $Z_\mu: \Hp\times\Hp \rightarrow \psl$. 
If $x, y \in \Hp$, consider the geodesic arc $[x,y]$ and  let $m_1,\ldots, m_n$ be the geodesics in the support of 
$ \mu$ intersecting $ [ x, y]$ with  atomic measures $a_1,\ldots a_m \in \R_+$. Then we define the {\em bending cocycle}
$$Z_\mu(x,y) = R(m_1,ia_1)R(m_2,ia_2)\ldots R(m_{n-1},ia_{m-1})R(m_n,ia_n).$$
If $x \in m_1$ (or $y \in a_m$), we replace $a_1$ by $a_1/2$ (similarly $a_m$ by $a_m/2$).
Then the map $f_\mu$ is defined by fixing a basepoint $x_0$ and defining
$$f_\mu(x) = Z_\mu(x_0,x)x.$$
This is well-defined up to post-composition by an element of $\psl$. We will denote  the image of  $f_\mu$ by $P_\mu$. 
This definition can be extended to all measured laminations by taking limits (see \cite[Chapter II.3]{EM}).  

If a locally convex pleated plane $P_\mu$ is embedded we call it a {\em convex pleated plane} and it bounds a convex region $X_\mu$ in
$\mathbb H^3$. If $f_\mu$ extends continuously to an embedding
$\bar f_\mu:\partial\mathbb H^2\to\partial\mathbb  H^3$, we let  $\Omega_\mu$ denote the Jordan domain of $\partial\mathbb H^3$ facing
the boundary $P_\mu$  of the convex region $X_\mu$.  Notice that if $\beta_\pm$ are the bending laminations of $\rho\in \QF(S)$, 
then the boundary $\partial CH(\rho)$ consists of
two convex pleated planes whose bending laminations are the lifts of $\beta_\pm$ to $\mathbb H^2$.

\subsection{$\theta$-bounded pleated planes}
We now give a concrete definition of the angle we want to bound when $P_\mu$ is a convex pleated plane.

Let $\alpha:[0,\infty)\to\Hp$ be  a unit speed geodesic ray and $\beta = f_\mu\circ \alpha$.   The map $\beta$ may not be differentiable at every point,
but for all $t$ it has a left and right derivative $\beta'_\pm(t)$. If $t>0$, we let $\gamma_t:[0,\infty)\to \mathbb H^3$ such that
$\gamma_t(0)=\alpha(0)$ and $\gamma_t(r_t)=\beta(t)$ where $r_t=d(\beta(0),\beta(t))$. We then define
$\theta^\pm_\mu(t)$ to be the angle at $\beta(t)$ between the tangent vectors $\gamma_t'(r_t)$
and $\beta_\pm'(t)$ at the point $\beta(t)$.
If $\beta(t)$ does not lie on an isolated leaf of $\mu$, then $\theta_\alpha^-(t) = \theta^+_\alpha(t)$. 
Furthermore $\theta_\mu^-$ is continuous from the left and $\theta^+_\mu$ is continuous from the right.

Given $p\ne q\in\Hp$, let $\alpha:\R\to \Hp$ be the unique unit speed geodesic ray with $\alpha(0) = p$ and $\alpha(d(p,q)) = q$. We define
$$\theta_\mu(p,q) = \theta^+_\mu(p,q) = \theta^+_\alpha(d(p,q)) \qquad \theta^-_\mu(p,q) =  \theta^-_\alpha(d(p,q)).$$
If $q$ is not on an isolated leaf of $\mu$, then $\theta^-_\mu(p,q) = \theta^+_\mu(p,q)= \theta_\mu(p,q)$. 

Given $\theta \in [0,\pi)$ we define $\mu$ to be {\em $\theta$-bounded} if $\theta_\mu(p,q) \le \theta$ for all $p\ne q\in\Hp$. Notice that  continuity
properties of $\theta_\mu^{\pm}$ imply that $\mu$ is $\theta$-bounded  if and only if 
$\theta^{-}_\mu(p,q) \le \theta$ for all $p\ne q\in\Hp$.
 
We first observe that if $\mu$ is $\theta$-bounded for some $\theta<\frac{\pi}{2}$, then $f_\mu$ is a bilipschitz embedding.

\begin{lemma}
If $\mu\in \mathcal{ML}(\Hp)$ is $\theta$-bounded for some $\theta <\frac{\pi}{2}$, then $f_\mu:\mathbb H^2 \rightarrow \mathbb H^3$ is a $(\sec\theta)$-bilipschitz
embedding. In particular $f_\mu$ extends continuously to an embedding
$\bar f_\mu:\partial\mathbb H^2\to\partial \mathbb H^3$.
\end{lemma}

\begin{proof}
By definition $f_\mu$ is $1$-Lipschitz so we only need to establish the lower bound.  We assume first that $\mu$ has locally finite support. 
Let $x,y\in \mathbb H^2$  and $\alpha:[0,T]\rightarrow \mathbb H^2$ be a geodesic parametrized by arc length with 
$x = \alpha(0)$ and $y=\alpha(T)$. We let $s(t) = d(f_\mu(x), f_\mu(\alpha(t))$ and let $t_i$ be values of $t$ such that $f_\mu(\alpha(t))$  
is on a  bending line. If $f_\mu\circ \alpha$ does not intersect  $\mu$ transversely, then it is a geodesic and $f_\mu$ is an isometry on $\alpha$. 
Otherwise, there is a finite collection $\{t_i\}\subset [0,T]$ so that $f_\mu\circ \alpha$ is geodesic on $[t_i,t_i+1]$.
If $t\neq t_i$  we define $\theta(t) = \theta_\mu(x, \alpha(t))$. Then a simple calculation gives that
$$s'(t) = \cos(\theta(t)) \geq \cos(\theta).$$
It follows that $s(t) \geq \cos(\theta)t$, so
$$d(f_\mu(x), f_\mu(y)) = s(T) \geq \cos(\theta)T = \cos(\theta)d(x,y).$$

For a general $\theta$-bounded lamination $\mu$, we consider the path $\gamma(t) = f_\mu(\alpha(t))$. 
Then $\gamma$ is a $\theta$-bounded curve. Let $\mathcal P_n$ be an evenly spaced partition of 
$[0,T]$ with $n$ vertices and let $\gamma_n:[0,T]\rightarrow \mathbb H^3$ 
be the piecewise geodesic path with vertices $f_\mu(\alpha(\mathcal P_n))$. 
Then, $\gamma_n$ is a $\theta_n$ bounded curve and $\theta_n \rightarrow \theta$, so by taking limits we see that
$$d(f_\mu(x), f_\mu(y)) \geq \cos(\theta)d(x,y).$$
\end{proof}

We now prove that $\theta$-bounded Jordan domains are moderately bent when $\theta<\frac{\pi}{2}$.
 
\begin{proposition}\label{theta-mod}
If $\mu\in \mathcal{ML}(\Hp)$ is $\theta$-bounded for some $\theta <\frac{\pi}{2}$, then $\Omega_\mu$ is moderately bent.
\end{proposition}

\begin{proof}
Let $(x,y)$ be a bending pair for $\Omega_\mu$. We may assume that  $(x,y) =(0,\infty)$ and that the plane $Q$ lying above the real axis $\R$
is a support plane for the pleated plane $P_\mu$. We can assume that $\Omega_\mu$ contains the upper half-plane.  
Let $P_-$  and $P_+$ be the left and right half-planes for the geodesic $\overline{0\infty}$.
(Notice that if $g$ is not an atom for the measure, then $Q=P_-=P_+$.)

Suppose that $u\in \partial \Omega_\mu\setminus\{0,\infty\}$ lies to the right of $0$.
Let $\eta:\R\to \mathbb H^3$ be the unique geodesic ray so that $\eta(0)=p=(0,0,|u|)\in\overline{0\infty}$ and
$\lim_{t\to +\infty} \eta(t)=u$. Notice that $\eta(\R)$ is perpendicular to $\overline{0\infty}$. 
Let $\gamma:\R\to P_\mu$ be a geodesic  in the intrinsic metric on $P_\mu$ so that  $p=\gamma(0)$  and  $\lim_{t\to +\infty} \gamma(t)=u$. 
Let $\alpha_t:\R\to \mathbb H^3$ be a geodesic  in $\mathbb H^3$ so that  $p=\alpha_t(0)$ and $\alpha_t(d(p,\gamma(t)))=\gamma(t)$.
Since $\mu$ is $\theta$-bounded, 
the angle  $\theta_\mu(\gamma(t),\gamma(0))$ is at most $\theta$. By definition, $\theta_\mu(\gamma(t),\gamma(0))$ is the angle  between 
$-\alpha_t'(0)$  and $-\gamma_-'(0)$, or equivalently,  the angle  between $\alpha_t'(0)$  and $\gamma_-'(0)$.

Since $f_\mu$ extends continuously to an embedding 
$\bar f_\mu:\partial\mathbb H^2\to\partial \mathbb H^3$,  we see that $\alpha_t$ converges to $\eta$, so 
$\alpha_t'(0)$ converges to $\eta'(0)$. Therefore, the angle between $\eta'(0)$ and  $\gamma_-'(0)$  is at most $\theta$. 

Notice that $\gamma_-'(0)$ lies in the portion of $P_-$ which is to the right of $\overline{0\infty}$. We see, by projecting $\eta$ onto the plane that the
Euclidean line segment $R_u$ joining 0 to u makes an angle of at most $\theta$ with portion $R_-$ of $\partial P_-$ which lies to the right of $0$.
Notice that $R_-$ must be a line segment with positive slope, since otherwise the  bending at $\overline{0\infty}$ would be at least $\frac{\pi}{2}$, which would contradict  the fact that
$\mu$ is $\theta$-bounded. Since $R_u$ lies in the lower half plane and makes an angle less than $\frac{\pi}{2}$ with $R_-$ it must have negative slope.
Therefore, $\Re(u)>0$ (see Figure \ref{fig:theta}).

We may similarly show that if $u \in \partial\Omega_\mu$ lies to the left of $0$, then $\Re(u)<0$.
Therefore, if $L$ is the imaginary axis, then $L$ is a round circle in $\hat\C$ which intersects $\partial\Omega_\mu$ transversely at the points
$0$ and $\infty$ and intersects no other point of $\partial\Omega_\mu$.
\begin{figure}[htbp] 
   \centering
   \includegraphics[width=3in]{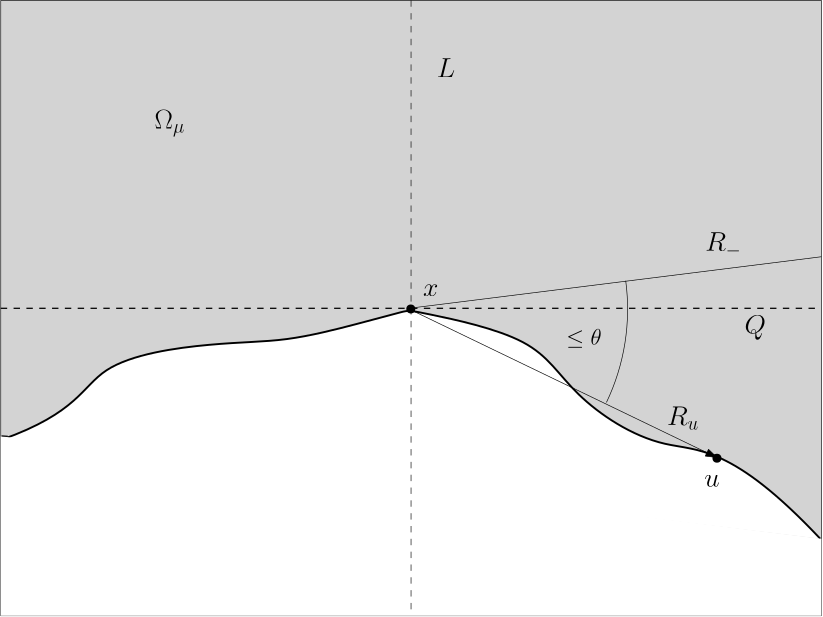} 
   \caption{Plane $L$}
   \label{fig:theta}
\end{figure}
\end{proof}

\subsection{Bending bounds}
We now give explicit bounds on $\|\mu\|_L$ such that the associated pleated plane is $\theta$-bounded and therefore the associated Jordan domain  is moderately bounded.

In earlier work \cite{BCY} the first two authors and  Yarmola produced an explicit upper bound on  $\|\mu\|_L$ which guarantees that $f_\mu$ is a bilipschitz embedding.

\begin{theorem}{\rm (Bridgeman-Canary-Yarmola \cite{BCY})}
\label{BCY}
There exists an explicit function $G:\R_+\rightarrow\R_+$ such that if $\mu$ is a uniformly bounded measured lamination on
$\mathbb H^2$ and $\|\mu\|_L < G(L)$, then $f_\mu$ is an  bilipschitz embedding. In particular,  $P_\mu$ is a convex pleated plane
and $f_\mu$ extends continuously to an embedding
$\bar f_\mu:\partial\mathbb H^2\to\partial \mathbb H^3.$
\end{theorem}

For all $\theta\in (0,\frac{\pi}{2}]$, we will obtain a bound on $||\mu||_L$ which guarantees that $\mu$ is $\theta$-bounded.
Our bounds will be defined using the hill function $h:\R\rightarrow (0,\pi)$ given by
$$h(t) = \cos^{-1}(\tanh(t))\quad\text{so}\quad h'(t)=-\sech(t)=-\sin h(t)\quad\text{and}\quad h(0)=\frac{\pi}{2}.$$
This function arises naturally in our situation since if $\alpha:\R\to\mathbb H^2$ is geodesic,
$x_0\in\mathbb H^2$ does not lie on the geodesic, $s(t)=d(x_0,\alpha(t))$, the angle 
$\theta(t)$ between the geodesic ray $\overrightarrow{x_0\alpha(t)}$ and $\alpha'(t)$ is monotonically decreasing,
then
$$s'(t)=\cos \theta(t)  \quad\text{and}\quad \theta'(t)=-\frac{\sin(\theta(t))}{\tanh s(t)}\le -\sin(\theta(t))$$
(see \cite[Lemma 4.4]{EMM}). 
We can then define a function $g:\R\to\R$ so that $h(g(t))=\theta(t)$, so that
the pair $(g(t), h(t))$ lies on the graph of the hill function.

Now consider the function
$$u_L(x) = h(x) -Lh'(x)$$
and notice that
$$\lim_{x\to +\infty} u_L(x)=0\quad\text{and}\quad \lim_{x\to -\infty} u_L(x)=\pi,$$
since $\lim_{x\to\pm\infty}h'(x)=0$, $\lim_{x\to+\infty} h(x)=\pi$ and $\lim_{x\to-\infty} h(x)=0.$
One may check that
$$u_L'(x) = -\sech(x) +L\sech(x)\tanh(x) = -\sech(x)(1-L\tanh(x)),$$  
so, if $L\in(0,1]$, then $u_L(x)<0$ for all $x\in\R$. Thus, if $L\in(0,1]$,
$u_L$ is a strictly decreasing function with image $(0, \pi)$.
So, given any $\theta \in [0,\pi/2]$ there is a unique value $a_L(\theta)$ so that $u_L(a_L(\theta)) = \theta$ (see Figure \ref{fig:hill-r}).  
We define $r_L(\theta)$ by the equation
$$r_L(\theta) = -Lh'(a_L(\theta))= L\sech(a_L(\theta)).$$  
If $L>1$, we define
$$r_L(\theta) =  -Lh'(L+h^{-1}(\theta)) = L\sech(L+\tanh^{-1}(\cos(\theta)).$$
In all cases, we define  $r(L) = r_L(\pi/2)$. 

We will give a more explicit description of these functions at the end of this section.

\begin{figure}[htbp] 
   \centering
   \includegraphics[width=4in]{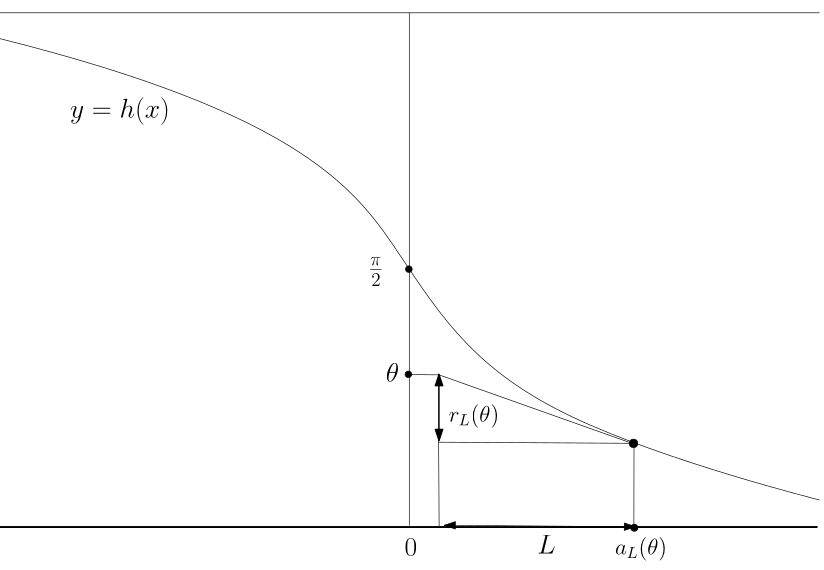} 
   \caption{Functions $a_L(\theta)$ and $r_L(\theta)$ for $L \leq 1$}
   \label{fig:hill-r}
\end{figure}

\medskip

We obtain the following strengthening of Theorem \ref{BCY}.

\begin{theorem}{($\theta$-bounded criterion)}
\label{theta bounded}
Let $\theta \in [0,\pi/2]$ and $L > 0$. If $\mu\in \mathcal{ML}(\Hp)$ and $\|\mu\|_L < r_L(\theta)$, 
then $P_\mu$ is a convex pleated plane and $\mu$ is  $\theta$-bounded. In particular, if $\|\mu\|_L < r(L)$,
then $P_\mu$ is $\theta$-bounded for some $\theta<\frac{\pi}{2}$.
\end{theorem}

Theorem \ref{Main Theorem 5} follows immediately from Proposition \ref{theta-mod} and Theorem \ref{theta bounded}.
We will discuss the proof of Corollary \ref{neighborhoods} at the end of the section.

\begin{proof}
We first show that, $r_L(\theta) < G(L)$ where $G$ is the function in Theorem \ref{BCY}. 
This guarantees that $P_\mu$ is a convex pleated plane.
We recall from \cite{BCY} that
$$G(L) = h(x_0-L)-h(x_0) = -Lh'(x_0)$$ 
where  $x_0$  is the unique point such that the tangent line to $(x_0,h(x_0))$ on the graph of $h$  intersects the graph of
$h$ again at the point $(x_0-L, h(x_0-L))$. They further show that $x_0 \in (0,L).$  

Suppose that $L\le1$, Then, by definition,
$$u_L(x_0)=h(x_0)-Lh'(x_0)=h(x_0-L) > \frac{\pi}{2} \geq \theta = u_L(a_L(\theta)).$$
Since $u_L$ is strictly decreasing, $a_L(\theta)>x_0$.
Since $h'$ is increasing on $(0,\infty)$ (but negative),
$$r_L(\theta) =-Lh'(a_L(\theta))  < -L h'(x_0) = G(L).$$
If $L>1$,  we simply observe that since $x_0\in (0,L)$ and $h'$ is increasing on $[0,\infty)$, we have
$$r_\theta(L)\le r_{\pi/2}(L)=-Lh'(L)<-Lh'(x_0)=G(L).$$

We also recall, see \cite[Lemma 4.3]{BCY}, that
if $\|\mu\|_L < G(L)$, then if $p,q\in\Hp$, then
$$\theta_\mu(p,q)\le h(x_0-L)<\pi.$$

We first assume that $\mu$ is a uniformly bounded measured lamination on $\mathbb H^2$ with locally finite support $\mathrm{supp}(\mu)$.
Let $\alpha:[0,\infty)\to\Hp$ be a unit-speed geodesic and  let $\beta=f_\mu\circ \alpha$.
Let $\{t_i\}_{i=1}^r=\alpha^{-1}(\mathrm{supp}(\mu))$ (where $r\in\mathbb N\cup\{\infty\}$) ordered so that $t_{i+1}>t_i$ for all $i$.

For all $t>0$, we define
$$\theta(t)=\theta^+(t)=\theta_\mu(\alpha(0),\alpha(t))\quad\text{and let}\quad\theta^-(t)=\theta^-_\mu(\alpha(0),\alpha(t)).$$
Notice that if $t\ne t_i$ (for some $i$), then $\theta^-(t)=\theta(t)$ and observe that
$$u_i = |\theta^+(t_i)-\theta^-(t_i)| \leq \phi_i=\mu(m(\alpha(t_i)))$$
where $\mu(m(\alpha(t_i)))$ is the measure of the leaf of $\mu$ passing through $\alpha(t_i)$.
One may calculate (see \cite[Lemma 4.4]{EMM}) that if $s(t) = d(\alpha(0),\alpha(t))$ and $t\ne t_i$ for any $i$, then
$$s'(t) = \cos(\theta(t)) \qquad \theta'(t) = \frac{-\sin(\theta(t))}{\tanh(s(t))} \leq -\sin(\theta(t)).$$
Thus on the intervals $(t_i,t_{i+1})$, the angle $\theta(t)$ is monotonically decreasing from $\theta^+(t_i)$ to $\theta^-(t_{i+1})$. 

We define a function $g:(t_1,\infty)\setminus \{t_i\}_{i=1}^r\to\R$  by the property that $h(g(t)) = \theta(t)$. 
This is well-defined and continuous  since $h$ is strictly monotone, continuous and has image $(0,\pi)$.
Similarly we define $g^\pm(t_i)$ so that $h(g^\pm(t_i)) = \theta^\pm(t_i)$.
This definition guarantees that $H(t)=(g(t),\theta(t))$ and $H^\pm(t_i)=(g^\pm(t_i),\theta^\pm(t_i))$ all lie on the graph of the hill function.

On $(t_i,t_{i+1})$ the image of $H(t)$ slides downward and to the right. 
Notice that if $t\in (t_i,t_{i+1})$ and $\theta(t)\in \left(0,\frac{\pi}{2}\right)$, then
 $$h'(g(t))g'(t) = \theta'(t) < -\sin(\theta(t)) = -\sin(h(g(t)) = h'(g(t)),$$
so $g'(t) > 1$. In particular, if $\theta(t_i)<\frac{\pi}{2}$, then
$$g^+(t_{i+1}) - g^-(t_i) > t_{i+1}-t_i.$$
At each $t_i$, the function transitions from $(g^-(t_i),\theta^-(t_i))$ to
$(g^+(t_i),\theta^+(t_i))$, so it jumps upward by $u_i \le \phi_i$ and to the left by  $g^-(t_i)-g^+(t_i)$.

First, suppose that $L\le 1$ and $\theta\le\frac{\pi}{2}$.
Let $a = a_L(\theta)$ and notice that, since  $h(a) -Lh'(a) = \theta$, we must have $h(a)<\theta$. 
Assume that $\theta(t) \geq \theta$ for some $t$ and let
$$T_1 = \inf\{ t \in (t_1,\infty) \ | \ \theta(t) \geq \theta\}\quad\text{and}\quad T_0 = \sup\{ t \in (t_1,T_1) \ | \ \theta^-(t) \leq h(a)\}.$$
Notice that since $\theta^-$ is continuous from the left, $\theta^-(T_0)\le h(a)$.

Since there is a total upward shift less than $r_L(\theta)$ in the interval $[T_0,T_0+L)$, we see that 
$$\theta(s) < h(a)+r_L(\theta) = h(a) -Lh'(a) = \theta\le\frac{\pi}{2}$$
for all $s\in[T_0,T_0+L)$. Thus, 
 $T_1 \geq T_0+L$. 

We can make the argument above more concrete in the following fashion. Let $\{t_j,\ldots,t_k\}=\{t_i\}\cap (T_0,T_0+L)$.
If this set is non-empty, let $T_0=t_{j-1}$ and $T_0+L=t_{k+1}$, while if the set is empty let $T_0=t_{j-1}$ (for some arbitrary choice of $j$) and $T_0+L=t_j=t_{k+1}$.
In either case, $\{t_{j-1},\ldots,t_{k+1}\}$ is a partition of $[T_0,T_0+L]$.
If $s\in [t_{\ell},t_{\ell+1})$ for some $\ell\in\{j-1,\ldots,k\}$, then
\begin{eqnarray*}
\theta^-(s)&=&\theta^{-}(T_0)+ \left(\sum_{i=j-1}^\ell \theta^+(t_i)-\theta^-(t_i)\right)+\left(\sum_{i=j-1}^{\ell-1}  \theta^-(t_{i+1})-\theta^+(t_i)\right)+\theta(s)-\theta(t_\ell)\\
& < & h(a)+\left(\sum_{i=j-1} \theta^+(t_i)-\theta^-(t_i)\right)\\
&< &  h(a)+r_L(\theta).\\
\end{eqnarray*}
Here we use the facts that  the third term  is  strictly negative and the fourth term is non-positive (since $\theta(t)$ is strictly decreasing on each interval) and that
the second term is bounded above by $r_\theta(L)$ by assumption.

To complete the proof for $L \leq 1$, we now show that $\theta^-(T_0+L) \leq h(a)$ which contradicts our definition of $T_0$. Since $g'(t)>1$ if $t\ne t_i$ and  $\theta(t)\le\frac{\pi}{2}$, there  is a horizontal shift right of at least $L$ on the interval $[T_0,T_0+L)$. 
Since $h'$ is decreasing in $[0,a]$, the total downward shift (associated with the horizontal shift)
is greater than $h'(a)L$. Since the total upward shift is less than $r_L(\theta)=-Lh'(a)$, we conclude that  $\theta^-(T_0+L) \leq h(a)$ giving us our contradiction. 

Using the notation above, we can again make this more concrete.
Consider the interval $[t_\ell, t_{\ell+1})$ for some $\ell \in \{t_{j-1},\ldots,t_{k}\}$. Since  $h(g^\pm(t_i)) = \theta^\pm(t_i)$ then by definition of $T_0$ the interval $[g^+(t_\ell),g^-(t_{\ell+1})]$ is contained in  $[0,a]$. Since $h'$ is decreasing on $[0,\infty)$ and $g'(t) > 1$ on $(t_{\ell},t_{\ell+1})$, the mean value theorem 
implies that
$$\theta^-(t_{\ell+1})-\theta^+(t_\ell) \leq h'(a)(g^-(t_{\ell+1})-g^+(t_\ell)) < h'(a)(t_{\ell+1}-t_\ell).$$
Therefore
\begin{eqnarray*}
\theta^-(T_0+L)&=&\theta^{-}(T_0)+ \left(\sum_{i=j-1}^k \theta^+(t_i)-\theta^-(t_i)\right)+\left(\sum_{i=j-1}^{k}  \theta^-(t_{i+1})-\theta^+(t_i)\right)\\
& <& h(a)+\left(\sum_{i=j-1}^k \theta^+(t_i)-\theta^-(t_i)\right) + h'(a)\left(\sum_{i=j-1}^{k}  t_{i+1}-t_i\right)\\
&\le&  h(a)+r_L(\theta)+h'(a)L \\
&\le &   h(a)\\
\end{eqnarray*}

The argument follows the same outline if $L>1$ and $\theta\le\frac{\pi}{2}$. 
We define $b = L+h^{-1}(\theta)$, or equivalently, by the equation $h(b-L) = \theta$, so $r_L(\theta)=-Lh'(b)$. As $h$ is decreasing  and concave on $[0,\infty)$  
$$h(b)<\theta\quad\text{and}\quad h(b)-Lh'(b)<h(b-L)=\theta.$$ 
Assume that $\theta(t) \geq \theta$ for some $t$ and let
$$T_1 = \inf\{ t \in (t_1,\infty) \ | \ \theta(t) \geq \theta\}\quad\text{and}\quad T_0 = \sup\{ t \in (t_1,T_1) \ | \ \theta^-(t) \leq h(b)\}.$$
If $s\in[T_0,T_0+L)$, then 
$$\theta(s) < h(b)+r_L(\theta) = h(b) -Lh'(b) <h(b-L)= \theta,$$
so $T_1\ge T_0+L$. We argue as before that
$$\liminf_{s\rightarrow T_0+L} \theta(s) \le h(b) -Lh'(b)+Lh'(b) = h(b)$$
which again contradicts our definition of $T_0$ and completes the proof when $L>1$.

If $\mu$ is any measured lamination  on $\Hp$ with $\|\mu\|_L < r_L(\theta)$ then
by Epstein-Marden-Markovic \cite[Lemma 4.6]{EMM2} there exists a sequence $\mu_n$  of finite-leaved measured laminations converging to $\mu$ 
such that $\|\mu_n\|_L = \|\mu\|_L$ for all $n$. By \cite[Corollary 4.8]{BCY}  this implies that each map $f_{\mu_n}$ is a $K$-bilipschitz embedding  for 
some $K$ depending only on $L$ and $\|\mu\|_L$. Then by \cite[Theorem III.3.11.9]{EM} the maps $f_{\mu_n}$ converge uniformly on compact sets 
to $f_\mu$. It follows that as $f_{\mu_n}$ are $\theta$-bounded, that $f_\mu$ is also $\theta$-bounded.  
\end{proof}

\medskip\noindent
{\bf Remark:}  The argument we used for $L>1$ works for all $L$. However,
 if $L \leq 1$ and $\theta\le\frac{\pi}{2}$, then $0 < a_L(\theta)< L+h^{-1}(\theta)$, so the argument we give when $L\le 1$, yields a better bound.

\subsection{Explicit description of bounds}
For $L>1$,
$$r_L(\theta) =  -Lh'(L+h^{-1}(\theta)) = L\sech(L+\tanh^{-1}(\cos(\theta)).$$
It follows that
$$r(L) = r_L(\pi/2) = L\sech(L).$$
For $L \leq 1$
$$r_L(\theta) = -Lh'(a_L(\theta)) = L\sech(a_L(\theta))$$
then
$$a_L(\theta) = \cosh^{-1}\left(\frac{L}{r_L(\theta)}\right).$$ 
As $a_L(\theta)$ is defined by the equation $\theta = u_L(a_L(\theta))$ and
$$u_L(x) = h(x) -  Lh'(x) = \cos^{-1}(\tanh(x))+L\sech(x)$$
then
$$\theta = u_L\left(\cosh^{-1}\left(\frac{L}{r_L(\theta)}\right)\right) = \cos^{-1}\left(\sqrt{1-\frac{r_L(\theta)^2}{L^2}}\right)+r_L(\theta).$$
Equivalently
$$\theta =  \sin^{-1}\left(\frac{r_L(\theta)}{L}\right)+r_L(\theta).$$
\begin{figure}[htbp] 
   \centering
   \includegraphics[width=3in]{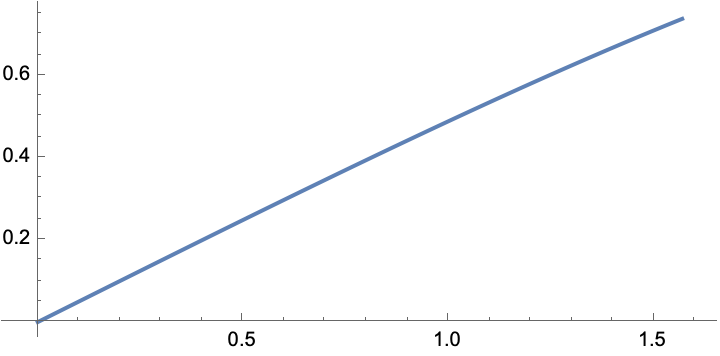} 
   \caption{Plot of $r_1(\theta)$ on $[0,\frac{\pi}{2}]$}
   \label{fig:r_1}
\end{figure}
Thus $r_L$ is the inverse function of $y = \sin^{-1}(x/L)+x$.  Therefore $r_L(\theta)$ is the solution to the equation
$$x = L\sin(\theta-x).$$
As $r(L) = r_L(\pi/2)$, it follows that $r$ is the inverse function of the function 
$$L(x) = x\sec(x).$$
In particular for $L=1$, 
$$r(1) = r_1(\pi/2) \simeq 0.739085.$$

\subsection{Proof of Corollary \ref{neighborhoods}}
Theorems \ref{Main Theorem 2} and \ref{Main Theorem 5} together imply that if $\rho\in U(S)$ is not fuchsian, then $\rho$ is not a critical point
of the entropy function.
Theorems \ref{Main Theorem 3} and \ref{Main Theorem 5} together imply that if $\rho\in V(S)$ is not fuchsian, then $\Ad\rho$ is the linear part of
a proper affine action of $\pi_1(S)$ on $\sl(2,\C)$. It remains to show that $US)$ and $V(S)$ are open.

We first observe that if $\alpha:\R \to \mathbb H^2$ is a geodesic and $\mu$ is a geodesic lamination on $\mathbb H^2$, then
$$\mu\left(\alpha([a, b))\right)=\lim_{\epsilon\to 0^+}\mu\left(\alpha((a-\epsilon,b-\epsilon))\right).$$
It follows that $||\mu||_L$ can also be defined to be the supremum of the transverse measure of any open geodesic arc of length $L$.
One can then easily check that $||\cdot||_L$ is a continuous function on $\mathcal T(S)\times \mathcal{ML}(S)$.
Keen and Series \cite{keen-series} showed that the map from $QF(S)$ to $\mathcal T(S)\times \mathcal{ML}(S)$ which takes $\rho$
to $(X^\nu(\rho),\beta_\nu)$ is continuous for both $\nu\in\{\pm\}$. It follows that $U(S)$ and $V(S)$ are open subsets of $QF(S)$.
\qed

\section{Teichm\"uller distance, Schwarzian derivatives and quasicircles}
\label{classical}
In this section, we obtain versions of our main results in terms of more
classical invariants.

Our first description is expressed in terms of the Schwarzian derivative.
Given  a locally univalent map  $f:\Delta\rightarrow\hat\C$ and $z\in\Delta$, let $M_{f,z}$ be the M\"obius transformation with the same 
2-jet as $f$ at $z$. It follows that $M_{f,z}^{-1}\circ f$ has the same two-jet as the identity at $z$ and the {\em Schwarzian derivative} $S(f)$  of $f$  at $z$ is then given by
$$S(f)(z) = ( M_{f,z}^{-1}\circ f)'''(z).$$
In particular, if $f$ is  a M\"obius transformation, then $S(f) = 0$.
In general  the Schwarzian derivative measures of how close $f$ is to being M\"obius. 
The Schwarzian derivative is also given by the equation
$$S(f) = \left(\frac{f''}{f'}\right)'-\frac{1}{2} \left(\frac{f''}{f'}\right)^2.$$
The Schwarzian derivative is a quadratic differential on $\Delta$. Denoting the space of quadratic differentials on a Riemann surface $X$ by $Q(X)$ we define the norm on $Q(X)$ by 
$$\|\phi\|_\infty = \sup_{z\in X} \frac{|\phi(z)|}{\rho_X(z)}$$
where $\rho_X(z) |dz|^2$ is the hyperbolic metric on $X$. 
Classical results of Nehari \cite{nehari} show that if $f:\Delta\rightarrow \hat\C$ is univalent then $\|S(f)\|\leq 3/2$ 
and if $f:\Delta\rightarrow \hat\C$ is locally univalent and $\|S(f)\|\leq 1/2$, then $f$ is univalent.

Let $\rho \in \QF(S)$ with limit set $\Lambda(\rho)$ and complementary Jordan domains $\Omega_+(\rho)$ and $\Omega_-(\rho)$. We let $f_\nu(\rho):\Delta \rightarrow \Omega_\nu$ be the maps uniformizing $\Omega_\nu(\rho), \nu \in \{ \pm \}$ and define quadratic differentials  $\tilde\phi_\nu(\rho) = S(f_\nu(\rho))$ in $Q(\Delta)$. Letting $X_\nu(\rho) = \Omega_\nu/\rho(\pi_1(S))$, then $\tilde\phi_\nu(\rho)$ descends to a quadratic differential $\phi_\nu(\rho)  \in Q(X_\nu(\rho))$.

By Bers simultaneous uniformization the map $QF(S) \rightarrow \mathcal T(S)\times\mathcal T(\bar S)$ given by $\rho \rightarrow (X_+(\rho),X_-(\rho))$ is a diffeomorphism (see \cite{bers}). We let $(X,Y) \rightarrow \rho(X,Y)$ be the inverse of this map.

For $X \in \mathcal T(S)$  the {\em Bers embedding} is the map
$B_X:\mathcal T(\overline{S})\rightarrow Q(X)$ given by
$$B_X(Y) = \phi_+(\rho(X,Y)).$$

Nehari's bounds imply that the image of the Bers embedding $B_X$ in the Banach space $(Q(X), \|\cdot\|_\infty)$ is contained in the ball of 
radius 3/2 about $0$ and contains the ball of radius $1/2$ about $0$. We first show that for a smaller ball about zero, we obtain our criterion.

 \begin{theorem}
 \label{schwarzian bound}
If $\rho \in \QF(S)$ and $0 < \|\phi_\nu(\rho)\| < .0739$ for some $\nu \in \{\pm\}$, 
then $\rho$ is not a critical point of $h$ and 
there is a proper affine action of $\pi_1(S)$ on $\mathfrak{sl}(2,\CC)$ with linear part  $\Ad(\rho)$.
\end{theorem}

\begin{proof}
Let $\beta_\nu$ be the bending laminations of $\rho$.
Bridgeman and Tee \cite{BrTee} show that  if  $\|\phi_\nu(\rho)\|_\infty < \frac{1}{2}\frac{1}{\sqrt{1+e^{2L}}}$, then
$$ \|\beta_\nu\|_L \leq 2\tan^{-1}\left( \frac{2\|\phi_\nu(\rho)\|_\infty e^L}{\sqrt{1-4\|\phi_\nu(\rho)\|_\infty^2}}\right) = F_L(\|\phi_\nu(\rho)\|_\infty).$$
If  we choose $\epsilon>0$ so that $\epsilon < F_L^{-1}(r(L))$ for some $L$, then if $\|\phi_\nu(\rho)\|_\infty < \epsilon$ their result implies that
$$\|\beta_\nu\|_L < r(L).$$
Theorems \ref{Main Theorem 2}, \ref{Main Theorem 3}  and \ref{Main Theorem 5} then imply that $\rho$
is not a critical point of entropy and that there is a proper affine action on $\mathfrak{sl}(2,\CC)$ with linear part $\Ad\rho$. 

Recall that  $L(r) = r/\cos(r)$ when $r\in (0,r(1)]$ and define $G:(0,r(1)]\rightarrow \R$ by $G(r) = F_{L(r)}^{-1}(r)$. 
So if $r\le r(1)$ and  $\epsilon <  G(r)$, then $\epsilon < F_{L(r)}^{-1}(r(L))$ and Theorem \ref{schwarzian bound} holds for this value of $\epsilon$. 
Since  $r(1) \approx .73908$, we may choose $r = .611$ and compute that
$$G(.611) \approx .0739643$$
so our theorem holds with $\epsilon=.739$. 
\end{proof}

We now get a criterion involving the Teichm\"uller distance between $\Omega_+/\rho(\pi_1(S))$ and
$\Omega_0/\rho(\pi_1(S))$.

\begin{theorem}
\label{teich-thm}
 If $\rho \in \QF(S)$ and $0 < d_T\left(X_+(\rho),\overline{X_-(\rho)}\right) < .049$,
then $\rho$ is not a critical point of $h$ and 
there is a proper affine action of $\pi_1(S)$ on $\mathfrak{sl}(2,\CC)$ with linear part equal to $\Ad\rho$.
\end{theorem}

\begin{proof}
Let $\beta_\nu$ be the bending laminations of $\rho$.
Bridgeman and Tee \cite{BrTee} showed that the Bers embedding $B_{X}$ is $\frac{3}{2}$-Lipschitz with respect to the Teichm\"uller metric 
on the domain and $L^\infty$ norm on the image. 
Thus integrating along the Teichm\"uller geodesic between $\overline{X_-(\rho)}$ and $X_+(\rho)$ in $\mathcal T(S)$ we get
$$\|\phi_-(\rho)\|_\infty \leq \frac{3}{2}d_T\left(X_{+}(\rho),\overline{X_-(\rho)}\right).$$ 
Therefore, if $d_T\left(X_{+}(\rho),\overline{X_-(\rho)}\right) < .049$ then
$$\|\phi_-(\rho)\|_\infty < \frac{3}{2}(.049)  = . 0735,$$
so our result follows from Theorem \ref{schwarzian bound}.
\end{proof}

Our next criterion  is in terms of the quasiconformal distortion of the limit set. 
Recall that a Jordan curve is a {\em $K$-quasicircle} if it is the image of the unit circle under a 
$K$-quasiconformal homeomorphism of $\hat\C$. If $\rho \in \QF(S)$, then the limit set $\Lambda(\rho)$ of $\rho(\pi_1(S))$
is a $K$-quasicircle for some $K$.

\begin{theorem}\label{kqcircle}
If  $\rho\in \QF(S)$ is not fuchsian and its limit set $\Lambda(\rho$) is a $K$-quasicircle for some $K < 1.05$, then
 $\rho$ is not a critical point of $h$ and 
there is a proper affine action of $\pi_1(S)$ on $\mathfrak{sl}(2,\C)$ with linear part equal to $\Ad\rho$.
\end{theorem}

In the proof of Theorem \ref{kqcircle}, we will need a version of Theorem \ref{teich-thm} in the non-equivariant setting, i.e. in terms of
the distance between $\Omega_+(\rho)$ and $\Omega_-(\rho)$ in the universal Teichm\"uller space $\mathcal T$.

Roughly, the {\em universal Teichm\"uller space}, denoted  $\mathcal T$,  is the space of conformal structures on the unit disk (modulo boundary) which are
quasiconformal  to the standard conformal structure on the unit disk. 
Using the measurable Riemann mapping theorem,  this can be made explicit using Beltrami differentials as follows (see \cite{lehto} for further details). 

Let  $L^\infty(\Half)$  the space of measurable functions $\mu:\Half\rightarrow \C$ with essential supremum $\|\mu\|_\infty < \infty$ and $L^\infty_1(\Half)$ be the open unit ball about zero with respect to the $L^\infty$-norm.
If $\mu \in  L_1^\infty(\Half)$, let $\hat\mu\in L_1^\infty(\hat\C)$  be  the extension of $\mu$ to the lower half-plane by Schwarz reflection, i.e.
$$\hat\mu(z) = \left\{ \begin{matrix} 
\mu (z) & z\in \Half\\
&\\
\overline{\mu(\overline{z})}& z\in \overline{\Half}
\end{matrix}
\right.
$$
 We then define $w_\mu:\hat\C\rightarrow\hat\C$  to be  solution to the Beltrami equation $w_{\overline z} = \hat\mu w_z$  
which fixes $0$, $1$, and $\infty$.
Then $\mathcal T=  L_1^\infty(\Half)/\!\sim$ where $\mu \sim \nu$ if $w_\mu = w_\nu$ on $\R$. 
We denote by $d_{\mathcal T}$ the Teichm\"uller metric  on $\mathcal T$ given by
$$d_{\mathcal T}([\mu],[\nu]) = \frac{1}{2}\inf_h K(h)$$
where the infimum is over all $h:\Half\rightarrow \Half$  quasiconformal such that  $h = w_\nu\circ w_\mu^{-1}$ on $\R$. 

Similarly, we let $\overline{\mathcal T}$ be universal Teichm\"uller space of conformal structures quasiconformal to the standard conformal
structure on the lower half plane $\overline\Half$. Schwarz reflection  gives an isometry between $(\mathcal T,d_{\mathcal T})$ and $(\overline{\mathcal T}, d_{\overline{\mathcal T}})$.

For $[\mu] \in \overline{\mathcal T}$ we can  extend $\mu$ to $\check\mu\in L_1^\infty(\hat\C)$ where $\check\mu$ is zero on $\Half$ and let 
$w^\mu$ be the solution to the Beltrami equation  $w_{\overline z} = \check\mu w_z$ which fixes $0,1,\infty$. It follows that $f^\mu$ is conformal
on $\Half$. The Bers embedding $\beta_{\Half}:\overline{\mathcal T} \rightarrow Q(\Half)$ is then defined by  
$B_{\Half}([\mu])= S(\left.w^\mu\right|_{\Half})$.  

We note that if $\rho \in \QF(S)$ then by definition there exists a quasiconformal map 
$f: \hat\C \rightarrow \hat\C$ conjugating a fuchsian action on $\hat\C$ to the quasifuchsian action given by $\rho$. Therefore
$f$ maps $\Half$ to $\Omega_+(\rho)$ and $\overline{\Half}$ to $\Omega_-(\rho)$ and its Beltrami differential gives a point in 
$ \mathcal T\times \overline{\mathcal T}$ which we denote by $(\Omega_+(\rho), \Omega_-(\rho))$. 
The following is the desired generalization of Theorem \ref{teich-thm}.

\begin{theorem} 
\label{equivariant}
If $\rho \in \QF(S)$ and  $0 < d_{\mathcal T}\left(\Omega_+(\rho), \overline{\Omega_-(\rho)}\right) < .049$,
then $\rho$ is not a critical point of $h$ and 
there is a proper affine action of $\pi_1(S)$ on $\mathfrak{sl}(2,\CC)$ with linear part equal to $\Ad(\rho)$.
\end{theorem}

\begin{proof}
Bridgeman and Tee  \cite{BrTee} also prove that  the maps $B_\Half$ and  $B_{\overline\Half}$ are $\frac{3}{2}$-Lipschitz.
Thus integrating along the Teichm\"uller geodesic between $\overline{\Omega_-(\rho)} $ and $\Omega_+(\rho)$ in $\mathcal T$ we get
$$\|\tilde\phi_-(\rho)\|_\infty < \frac{3}{2}d_{\mathcal T}\left(\Omega_+(\rho), \overline{\Omega_-(\rho)}\right) .$$ 
So if $d_{\mathcal T}\left(\Omega_+(\rho), \overline{\Omega_-(\rho)}\right) < .049$, then
$$\|\phi_-(\rho)\|_\infty = \|\tilde\phi_-(\rho)\|_\infty < \frac{3}{2}(.049)  = . 0735$$
Then the result follows from Theorem \ref{schwarzian bound}.
\end{proof}

We can now prove Theorem \ref{kqcircle}.

\begin{proof}
Let $f:\hat\C\rightarrow \hat\C$ be the $K$-quasiconformal homeomorphism mapping $\Half$ to $\Omega_+(\rho)$ and $\overline\Half$ to $\Omega_-(\rho)$. Then
$$d_{\mathcal T}(\Omega_+(\rho), \Half) \leq \frac{1}{2}\log K(f)\qquad  d_{\overline{\mathcal T}}(\Omega_-(\rho),\overline\Half)  \leq \frac{1}{2}\log K(f).$$
Therefore 
$$d_{\mathcal T}(\Omega_+(\rho), \overline{\Omega_-(\rho)}) \leq d_{\mathcal T}(\Omega_+(\rho), \Half)+d_{\mathcal T}(\Half,\overline{\Omega_-(\rho)}) \leq \log K(f).$$
Thus if $K < e^{.049} =1.05022$ then $d_{\mathcal T}\left(\Omega_+(\rho), \overline{\Omega_-(\rho)}\right) \leq .049$ and the result follows  from Theorem \ref{equivariant} above.
\end{proof}

\end{document}